\documentclass[10pt]{article}
\usepackage{amssymb,amsfonts}
\usepackage{pb-diagram}
\usepackage{amsmath,amsthm,amsfonts,amssymb}
\usepackage[usenames]{color}
\usepackage{hyperref}
\usepackage{soul}
\usepackage{graphicx}
\usepackage{amssymb}
\usepackage{amsmath}
\usepackage{amssymb}

\newtheorem{theorem}{Theorem}
\newtheorem{definition}{Definition}
\newtheorem{example}{Example}
\newtheorem{lemma}{Lemma}
\newtheorem{corollary}{Corollary}
\newtheorem{remark}{Remark}
\newtheorem{conjecture}{Conjecture}
\newtheorem{proposition}{Proposition}
\newtheorem{problem}{Problem}
\newcommand{\be}{\begin{enumerate}}
\newcommand{\ee}{\end{enumerate}}
\newcommand{\beq}{\begin{equation}}
\newcommand{\eeq}{\end{equation}}
\newcommand{\ov}[1]{\mbox{${\overline{#1}}$}} 
\def\FZt{{F^{\mathbb{Z}[t]}}}
\def\Zt{{\mathbb{Z}[t]}}
\def\N{{\mathbb{N}}}
\def\Z{{\mathbb{Z}}}
\def\R{{\mathbb{R}}}
\def\bF{{\mathbb F}}

\DeclareMathOperator{\PGL}{PGL}

\newcommand{\tr}{{\rm tr}} 

\title{Beyond Serre's ``Trees"  in two directions:  $\Lambda$--trees and products of trees}
\author{Olga Kharlampovich\thanks{Hunter College and Grad. Center CUNY}, Alina Vdovina\thanks{Newcastle University and Hunter College CUNY}}
\begin{document}

\maketitle
\begin{abstract}
Serre \cite{Serre} laid down the fundamentals of the theory of
groups acting  on simplicial trees.  In particular, Bass-Serre theory makes it possible to
extract information about the structure of a group from its action on a simplicial tree. Serre's original motivation was to understand the structure of certain algebraic groups whose Bruhat--Tits buildings are trees.  In this survey we will discuss the following generalizations of  ideas from \cite{Serre}: the theory of isometric group actions on $\Lambda$-trees and the theory of lattices in the product of trees where we describe in more detail results on arithmetic groups acting on a product of trees.  
\end{abstract}

\section{Introduction}
Serre \cite{Serre} laid down the fundamentals of the theory of
groups acting  on simplicial trees. The book \cite{Serre} consists of two parts. The first part describes the basics of what is now called Bass-Serre theory. This theory makes it possible to 
extract information about the structure of a group from its action on a simplicial tree. Serre's original motivation was to understand the structure of certain algebraic groups whose Bruhat--Tits buildings are trees. These groups are considered in the second part of the book.  

Bass-Serre theory states that a  group acting on a tree can be decomposed (splits) as a  free product with amalgamation or an HNN extension. Such a group can be represented as a fundamental group of a graph of groups. This  became a wonderful tool in geometric group theory and geometric topology, in particular in the study of 3-manifolds.  The theory was further developing in the following directions: 
\begin{enumerate} 
\item Various accessibility results were proved for finitely presented groups that bound the complexity (that is, the number of edges) in a graph of groups decomposition of a finitely presented group, where some algebraic or geometric restrictions on the types of groups  were imposed  \cite{D,BF2,SA,Del,Weid}. \item The theory of JSJ-decompositions for finitely presented groups was developed \cite{Bowditch:98,RipsSels:97,DunwoodySageev,ScottSwarup,FujiwaraPapasoglu}.   \item The theory of lattices in automorphism groups of trees. The group of automorphisms of a locally  finite tree $Aut (T)$ (equipped with the compact open topology, where open neighborhoods of $f\in Aut (T)$ consist of all automorphisms that agree with $f$ on a fixed finite subtree) is a locally 
compact group which  behaves  similarly to a rank one simple Lie group.  This analogy has motivated many recent works  in particular the study of lattices in $Aut(T)$ by Bass,  Kulkarni,  Lubotzky \cite{BK}, \cite{Lu1} and others.  A survey of results about tree lattices and methods is given in \cite{BL}  as well as  proofs of many results.
\end{enumerate}
 In this survey we will discuss the following generalizations of  Bass-Serre theory and other Serre's ideas from \cite{Serre}:
\begin{enumerate}\item The theory of isometric group actions on real trees (or $\mathbb R$-trees) which are metric spaces generalizing the graph-theoretic notion of a tree. This topic will be only discussed briefly, we refer the reader to the survey \cite{Best}. \item  The theory of isometric group actions on
$\Lambda$-trees, see Sections \ref{sec:lambda}, \ref{S3}. Alperin and Bass \cite{AB} developed the initial framework and  stated the fundamental research goals: find the group theoretic information
carried by an action (by isometries) on a $\Lambda$-tree; generalize Bass-Serre theory to actions
on arbitrary $\Lambda$-trees. From the viewpoint of Bass-Serre theory, the question of free actions of finitely generated groups became very important.  There is a book \cite{Ch1} on the subject and many new results were obtained in \cite{KMS12}.  This is a topic of interest of the first author. 
 \item The theory of complexes of groups   provides a higher-dimensional generalization of Bass--Serre theory.  The methods developed for the study  of lattices in $Aut(T)$  were extended to the study of (irreducible) lattices in a product of two trees, as a first step toward generalizing  the theory of lattices in semisimple non-archimedean Lie  groups. Irreducible lattices in higher rank semisimple Lie  groups have a very
rich structure theory and there are superrigidity and arithmeticity theorems by Margulis.
The results of Burger, Moses, Zimmer  \cite{burger-mozes:simple, burger-mozes:lattices, BMZ} about cocompact lattices in 
 the group of automorphisms of a product of trees  or rather in groups of the form $Aut (T_1)\times\  Aut (T_2)$,  where each of the trees is regular, are described in \cite{mozes:survey}.  The results  obtained  concerning the structure of lattices in $Aut(T_1)\times\    Aut(T_2)$  enable them to construct the  first examples of  finitely presented torsion free simple groups.  We will  mention further results on simple and non-residually finite groups  \cite{wise1,rattaggi:simple, rattaggi:thesis} and describe in more detail results on arithmetic groups acting on the product of trees \cite{glasner-mozes,stix-av}. This is a topic of interest of the second author.
\end{enumerate}

 In \cite{Lyndon2} Lyndon introduced real-valued length functions as a tool to
extend Nielsen cancelation theory from free groups over to a more general setting.   Some results in this direction were obtained in \cite{Hoare1,Hoare2,Harrison,Prom, AM}.
The term ${\mathbb R}$-tree was coined by Morgan and Shalen
\cite{MS} in 1984 to describe a type of space that was first defined by Tits \cite{Tits:1977}. In \cite{Chiswell:1976}  Chiswell described a construction which shows that a group with a real-valued
length function has an action on an $\mathbb{R}$-tree, and vice versa. Morgan and Shalen
realized that a similar construction and results hold for an arbitrary group with a Lyndon length function which
takes values in an arbitrary ordered abelian group $\Lambda$ (see \cite{MS}). In particular, they
introduced $\Lambda$-trees as a natural generalization of $\mathbb{R}$-trees which they studied in
relation with Thurston's Geometrization Program. Thus, actions on $\Lambda$-trees and Lyndon length
functions with values in $\Lambda$ are two equivalent languages describing the same class of groups.
In the case when the action is free (the stabilizer of every point is trivial) we call groups in this
class $\Lambda$-free or tree-free. We refer to the book \cite{Ch1} for a detailed discussion on the subject.

A joint effort of several researchers culminated in a description of finitely generated groups acting
freely on $\mathbb{R}$-trees \cite{BF,GLP}, which is now known as Rips' theorem: a finitely
generated group acts freely on an $\mathbb{R}$-tree if and only if it is a free product of free
abelian groups and surface groups (with an exception of non-orientable surfaces of genus $1, 2$, and
$3$). The key ingredient of this theory is the so-called ``Rips machine'', the idea of which comes
from Makanin's algorithm for solving equations in free groups (see \cite{Mak82}). The Rips machine
appears in applications as a general tool that takes a sequence of isometric actions of a group $G$
on some ``negatively curved spaces'' and produces an isometric action of $G$ on an $\mathbb{R}$-tree
as the Gromov-Hausdorff limit of the sequence of spaces. Free actions on $\mathbb{R}$-trees cover
all Archimedean actions, since every group acting freely on a $\Lambda$-tree for an Archimedean ordered
abelian group $\Lambda$ also acts freely on an $\mathbb{R}$-tree.

In the non-Archimedean case the following results were obtained.
First of all, in \cite{B} Bass studied finitely generated groups acting freely on $\Lambda_0 \oplus
\mathbb{Z}$-trees with respect to the right lexicographic order on $\Lambda_0 \oplus \mathbb{Z}$, where $\Lambda _0$ is any ordered abelian group. In
this case it was shown that the group acting freely on a $\Lambda_0 \oplus \mathbb{Z}$-tree splits
into a graph of groups with $\Lambda_0$-free vertex groups and maximal abelian edge groups. Next,
Guirardel (see \cite{G}) obtained the structure of finitely generated groups acting freely on
$\mathbb{R}^n$-trees (with the lexicographic order). In \cite{KMRS1} the authors described the
class of finitely generated groups acting freely and regularly on $\mathbb{Z}^n$-trees in terms of
HNN-extensions of a very particular type. The action is {\em regular} if all branch points are in the same orbit. The importance of regular actions becomes clear from the
results of \cite{KMS2}, where it was proved that a finitely generated group acting freely on a
$\mathbb{Z}^n$-tree is a subgroup of a finitely generated group acting freely and regularly on a
$\mathbb{Z}^m$-tree for $m \geqslant n$, and the paper \cite{Ch2}, where it was shown that a group
acting freely on a $\Lambda$-tree (for arbitrary $\Lambda$) can always be embedded in a
length-preserving way into a group acting freely and regularly on a $\Lambda$-tree (for the same $\Lambda$).
The structure of finitely presented $\Lambda$-free groups was described in  \cite{KMS12}.  They  all are 
$\mathbb{R}^n$-free.

Another natural generalization of Bass-Serre theory is considering group actions
on products of trees started in \cite{burger-mozes:simple}.  The structure of a group acting freely and cocompactly on a simplicial tree is well
understood. Such a group is a finitely generated free group. By way of contrast,
a group which acts  similarly  on a product of trees can have remarkably
subtle properties. 

Returning to the case of one tree, recall that there is a close relation between certain simple Lie groups and groups of tree automorphisms.  The theory of tree lattices was developed  in \cite{BK}, \cite{Lu1} by analogy with the theory of lattices in Lie groups (that is discrete subgroups of Lie groups of finite co-volume). Let $G$ be a simple algebraic group of rank one over a non-archimedean local  field $K$.   Considering the action of $G$ on its associated Bruhat--Tits tree $T$   we have a continuous embedding of $G$ in $Aut (T)$ with co-compact image.  In  \cite{Tit}   Tits has shown that if $T$ is a locally  finite tree and its automorphism group $Aut (T)$ acts minimally  (i.e.  without an invariant proper subtree and not  fixing an end)  on it,  then the subgroup  generated by edge stabilizers is a simple group.  In particular the automorphism group of a regular tree is virtually simple.  These results motivated the study of $Aut (T)$ looking at the analogy with rank one Lie groups. 

When $T$ is a locally finite tree,  $G=Aut (T)$ is locally compact.  The vertex stabilizers $G_v$ are open and compact. A subgroup $\Gamma\leq G$ is {\em discrete} if $\Gamma _v$ is finite for some (and hence for every) vertex $v\in VT$, where $VT$ is the set of vertices of $T.$ In this case we can define  $$Vol (\Gamma \backslash T)=\sum _{v\in \Gamma \backslash VT}1/|\Gamma _v|.$$ We call $\Gamma$ a $T$-lattice if $Vol (\Gamma \backslash T)<\infty$. We call $\Gamma$ a {\em uniform $T$-lattice} if $\Gamma\backslash T$ is finite.  In case $G\backslash T$ is finite, this is equivalent to $\Gamma$ being a lattice (resp., uniform lattice) in $G$. Uniform tree lattices   correspond to  finite graphs of groups in which all vertex and edge groups are finite. 
In the study of lattices in semisimple Lie groups an important role is playied by their commensurators. Margulis has shown that an irreducible lattice $\Gamma<G$ in a semisimple Lie group is arithmetic if and only if its commensurator is dense in $G$. It was shown by Liu \cite{Liu} that the commensurator of a uniform tree lattice is dense.
All uniform tree lattices of a given tree are commensurable up to conjugation, the isomorphism class of the commensurator of a uniform tree lattice is determined by the tree.  It was also shown \cite{LMZ} that for regular trees the commensurator determines the tree.

Consider now products of trees.  The following results about  lattices in semisimple Lie groups were established by Margulis: 
an irreducible lattice in a higher rank  ($\geq 2$)  semisimple Lie group is arithmetic; 
any linear representation of such a lattice with unbounded image essentially extends to a continuous representation of the ambient Lie group. 
Recall that a lattice     $G$ in a semisimple Lie group is called reducible if $\Gamma$ 
     contains a  finite index subgroup of the form    $\Gamma _1\times\Gamma _2$ where $\Gamma _i<G_i$ is a lattice and $G=G_1\times G_2$.    A reducible torsion free  group acting simply transitively on a product of two trees  is virtually a direct product of two finitely generated free groups, in particular   is residually finite. Burger, Moses, Zimmer \cite{burger-mozes:simple} were studying a structure for lattices of groups of the form $Aut (T_1)\times\  Aut (T_2)$. For example, Burger and Mozes have proved
rigidity and arithmeticity results analogous to the theorems of Margulis for lattices
in semisimple Lie groups. We will describe their results  about cocompact lattices in 
 groups of the form $Aut (T_1)\times\  Aut (T_2)$,  where each of the trees is regular.
We will discuss this and similar topics in Section \ref{sec:products}.

The smallest explicit example of a simple group, as an index 4 subgroup of a group presented by 10 generators and 24 short relations, was constructed by Rattaggi \cite{rattaggi:simple}. For comparison, the smallest virtually simple group of \cite{burger-mozes:simple}, Theorem 6.4, needs more than 18000 relations, and the smallest simple group constructed in \cite{burger-mozes:simple}, §6.5, needs even more than 360000 relations in any finite presentation.

If we restrict our attention to torsion free lattices that act simply transitively on the vertices of the product of trees (not interchanging the factors), then those lattices are fundamental groups of square complexes with just one vertex, complete bipartite link and a vertical/horizontal structure on edges (see \ref{sec:VHstructure} for precise definition). This is combinatorially well understood and there are plenty of such lattices, see \cite{stix-av} for a mass formula, but very rarely these lattices arise from an arithmetic context. 

Let $T_n$ denote the tree of constant valency $n$. 
For example, there are $541$ labelled candidate square complexes that after forgetting the label give rise to $43$ torsion free and vertex transitive lattices acting on $T_4 \times T_4$. Among those only one lattice is arithmetic. For lattices acting on $T_6 \times T_6$ the number of labelled square complexes is $\approx 27 \cdot 10^6$, but only few thousands have finite abelianization (a necessary condition) and only $2$ lattices are known to be arithmetic.

For different odd prime numbers $p \not= \ell$, Mozes  \cite{mozes}, 
 for $p$ and $\ell$ congruent to $1$ mod $4$, 
and later, for any two distinct odd primes, Rattaggi  \cite{rattaggi:thesis} found an arithmetic lattice acting on $T_{p+1} \times T_{\ell +1}$ with  
simply transitive action on the vertices. 

For products of trees of the same valency arithmetic lattices were constructed in \cite{stix-av}.
By means of a quaternion algebra over $\mathbb{F}_q(t)$, there is an explicit construction of
an infinite  series of torsion free, simply transitive, irreducible lattices in $\PGL_2(\mathbb{F}_q((t))) \times \PGL_2(\mathbb{F}_q((t)))$.  The lattices depend  on an odd prime power $q = p^r$ and a parameter $\tau \in \mathbb{F}_q^\times, \tau \not= 1$, and are the fundamental groups of a square complex with just one vertex and universal covering $T_{q+1} \times T_{q+1}$, a product of trees with constant valency $q+1$. For $p=2$ there are examples of arithmetic lattices acting
on products of trees of valency 3 in \cite{nithi}.

The last subsection of our survey is dedicated to lattices in products of $n \geq 2$ trees.

 We would like to thank the referee  for giving constructive comments which substantially helped improving the quality of the exposition.

\section{$\Lambda$-trees}
\label{sec:lambda}

The theory of $\Lambda$-trees (where $\Lambda = \mathbb{R}$) has its origins in the papers by
Chiswell \cite{Chiswell:1976} and Tits \cite{Tits:1977}. The first paper contains a construction
of an $\mathbb{R}$-tree starting from a Lyndon length function on a group (see Section
\ref{sec:length_func}), an idea considered earlier by Lyndon in \cite{Lyndon:1963}.

\smallskip

Later, in their very influential paper \cite{Morgan_Shalen:1991} Morgan and Shalen linked group
actions on $\mathbb{R}$-trees with topology and generalized parts of Thurston's Geometrization
Theorem. Next, they introduced $\Lambda$-trees for an arbitrary ordered abelian group $\Lambda$ and
the general form of Chiswell's construction. Thus, it became clear that abstract length functions with
values in $\Lambda$ and group actions on $\Lambda$-trees are just two equivalent approaches to the
same realm of group theory questions. The unified theory was further developed in the important paper by Alperin and
Bass \cite{AB}, where authors state a fundamental problem in the theory of group
actions on $\Lambda$-trees: find the group theoretic information carried by a $\Lambda$-tree action
(analogous to Bass-Serre theory), in particular, describe finitely generated groups acting freely
on $\Lambda$-trees ($\Lambda$-free groups).

\smallskip

Here we introduce basics of the theory of $\Lambda$-trees, which can be found in more detail in
\cite{AB},  \cite{Ch1} and \cite{KMS12}.

\subsection{Ordered abelian groups}
\label{subs:ord_ab}

In this section some well-known results on ordered abelian groups are collected. For proofs and
details we refer to the books \cite{Glass:1999} and \cite{Kopytov_Medvedev:1996}.

A set $A$ equipped with addition ``$+$'' and a partial order ``$\leqslant$'' is called a {\em
partially ordered} abelian group if:
\begin{enumerate}
\item[(1)] $\langle A, + \rangle$ is an abelian group,
\item[(2)] $\langle A, \leqslant \rangle$ is a partially ordered set,
\item[(3)] for all $a,b,c \in A,\ a \leqslant b$ implies $a + c \leqslant b + c$.
\end{enumerate}

An abelian group $A$ is called {\em orderable} if there exists a linear order ``$\leqslant$'' on $A$,
satisfying the condition (3) above. In general, the ordering on $A$ is not unique.

Let $A$ and $B$ be ordered abelian groups. Then the direct sum  $A \oplus B$ is orderable with
respect to the {\em right lexicographic order}, defined as follows:
$$(a_1,b_1) < (a_2,b_2) \Leftrightarrow b_1 < b_2 \ \mbox{or} \ (b_1 = b_2 \ \mbox{and} \ a_1 < a_2).$$

Similarly, one can define the  right lexicographic order on finite direct sums of ordered abelian
groups or even on infinite direct sums if the set of indices is linearly ordered.

For elements $a,b$ of an ordered group $A$  the {\em closed segment} $[a,b]$ is defined by
$$[a,b] = \{c \in A \mid a \leqslant c \leqslant b\}.$$

A  subset $C \subset A$ is called {\em convex}, if for every $a, b \in C$ the set $C$ contains
$[a,b]$. In particular, a subgroup $B$ of $A$ is convex if $[0,b] \subset B$ for every positive $b
\in B$. In this event, the quotient $A / B$ is an ordered abelian group with respect to the order
induced from $A$.

A group $A$ is called {\em archimedean} if it has no non-trivial proper convex subgroups. It is
known  that $A$ is archimedean if and only if $A$ can be embedded into the ordered abelian group
of real numbers $\mathbb{R}_+$, or equivalently, for any $0 < a \in A$ and any $b \in A$ there
exists an integer $n$ such that $na > b$.

It is not hard to see that the set of convex subgroups of an ordered abelian group $A$ is linearly
ordered by inclusion (see, for example, \cite{Glass:1999}), it is called  {\em the complete chain of
convex subgroups} in $A$. Notice that
$$E_n = \{f(t) \in \Zt \mid {\rm deg}(f(t)) \leqslant n\}$$
is a convex subgroup of $\Zt$ (here ${\rm deg}(f(t))$ is the degree of $f(t)$) and
$$0 < E_0 < E_1 < \cdots < E_n < \cdots$$
is the complete chain of convex subgroups of $\Zt$.

If $A$ is finitely generated  then the complete chain of convex subgroups of $A$
$$0 = A_0 < A_1 < \cdots < A_n = A$$
is finite. The following result (see, for example, \cite{Ch1}) shows that this chain completely
determines the order on $A$, as well as the structure of $A$. Namely, the groups  $A_i / A_{i-1}$
are archimedean (with respect to the induced order) and $A$ is isomorphic (as an ordered group) to
the direct sum
\begin{equation}
\label{eq:order-convex}
A_1 \oplus A_2 / A_1 \oplus \cdots \oplus A_n / A_{n-1}
\end{equation}
with the right lexicographic order.

An ordered abelian group $A$ is called {\em discretely ordered} if $A$ has a non-trivial minimal
positive element (we denote it by $1_A$). In this event, for any $a \in A$ the following hold:

\begin{enumerate}
\item[(1)] $a + 1_A = \min\{b \mid b > a\}$,
\item[(2)] $a - 1_A = \max\{b \mid b < a\}$.
\end{enumerate}

For example, $A = \mathbb{Z}^n$ with the right lexicographic order is discretely ordered with
$1_{\mathbb{Z}^n} = (1, 0, \ldots, 0)$. The additive group of integer polynomials $\Zt$
is discretely ordered with $1_{\Zt} = 1$.

Recall that an ordered abelian group $A$ is {\em hereditary discrete} if for any convex subgroup
$E \leqslant A$ the quotient $A / E$ is discrete with respect to the induced order.
A finitely generated discretely ordered archimedean abelian group is infinite cyclic. A finitely generated hereditary discrete ordered abelian group is isomorphic
to the direct product of finitely many copies of $\mathbb{Z}$ with the lexicographic order. \cite{MRS2}

\subsection{$\Lambda$-metric spaces}
\label{subs:lambda-def}

Let $X$ be a non-empty set, and  $\Lambda$ an ordered abelian group. A {\em $\Lambda$-metric on $X$} is
a mapping $d: X \times X \longrightarrow \Lambda$ such that for all $x, y, z \in X$:
\begin{enumerate}
\item[(M1)] $d(x,y) \geqslant 0$,
\item[(M2)] $d(x,y) = 0$ if and only if $x = y$,
\item[(M3)] $d(x,y) = d(y,x)$,
\item[(M4)] $d(x,y) \leqslant d(x,z) + d(y,z)$.
\end{enumerate}

So a {\em $\Lambda$-metric space} is a pair $(X,d)$, where $X$ is a non-empty set and $d$ is a
$\Lambda$-metric on $X$. If $(X,d)$ and $(X',d')$ are $\Lambda$-metric spaces, an {\em isometry}
from $(X,d)$ to $(X',d')$ is a mapping $f: X \rightarrow X'$ such that $d(x,y) = d'(f(x),f(y))$ for
all $x, y \in X$.

A {\em segment} in a $\Lambda$-metric space is the image of an isometry $\alpha: [a,b]_\Lambda
\rightarrow X$ for some $a, b \in \Lambda$ and $[a,b]_\Lambda$ is a segment in $\Lambda$. The
endpoints of the segment are $\alpha(a), \alpha(b)$.

We call a $\Lambda$-metric space $(X,d)$ {\em geodesic} if for all $x, y \in X$, there is a segment
in $X$ with endpoints $x, y$ and $(X,d)$ is {\em geodesically linear} if for all $x, y \in X$,
there is a unique segment in $X$ whose set of endpoints is $\{x, y\}$.

It is not hard to see, for example, that $(\Lambda, d)$ is a geodesically linear $\Lambda$-metric
space, where $d(a,b) = |a - b|$, and the segment with endpoints $a,b$ is $[a,b]_\Lambda$.

Let $(X,d)$ be a $\Lambda$-metric space. Choose a point $v \in X$, and for $x,y \in X$, define
$$(x \cdot y)_v = \frac{1}{2} (d(x,v) + d(y,v) - d(x,y)).$$

Observe, that in general $(x \cdot y)_v \in \frac{1}{2} \Lambda$.

The following simple result follows immediately

\begin{lemma} \cite{Ch1}
If $(X,d)$ is a $\Lambda$-metric space then the following are equivalent:
\begin{enumerate}
\item for some $v \in X$ and all $x, y \in X, (x \cdot y)_v \in \Lambda$,
\item for all $v, x, y \in X, (x \cdot y)_v \in \Lambda$.
\end{enumerate}
\end{lemma}

Let $\delta \in \Lambda$ with $\delta \geqslant 0$. Then $(X,p)$ is {\em $\delta$-hyperbolic with
respect to $v$} if, for all $x, y, z \in X$,
$$(x \cdot y)_v \geqslant min \{(x \cdot z)_v, (z \cdot y)_v\} - \delta.$$

\begin{lemma}
\cite{Ch1}
If $(X, d)$ is $\delta$-hyperbolic with respect to $v$, and $t$ is any other point of $X$, then
$(X, d)$ is $2 \delta$-hyperbolic with respect to $t$.
\end{lemma}

A {\em $\Lambda$-tree} is a $\Lambda$-metric space $(X,d)$ such that:
\begin{enumerate}
\item[(T1)] $(X,d)$ is geodesic,
\item[(T2)] if two segments of $(X,d)$ intersect in a single point, which is an endpoint of both,
then their union is a segment,
\item[(T3)] the intersection of two segments with a common endpoint is also a segment.
\end{enumerate}

\begin{example}
$\Lambda$ together with the usual metric $d(a,b) = |a - b|$ is a $\Lambda$-tree. Moreover, any
convex set of $\Lambda$ is a $\Lambda$-tree.
\end{example}

\begin{example}
\label{examp:1}
A $\mathbb{Z}$-metric space $(X,d)$ is a $\mathbb{Z}$-tree if and only if there is a simplicial
tree $\Gamma$ such that $X = V(\Gamma)$ and $p$ is the path metric of $\Gamma$.
\end{example}

Observe  that in general a $\Lambda$-tree can not be viewed as a simplicial tree with the path
metric like in Example \ref{examp:1}.

\begin{lemma}
\cite{Ch1}
Let $(X,d)$ be $\Lambda$-tree. Then $(X,d)$ is $0$-hyperbolic, and for all $x, y, v \in X$ we have
$(x \cdot y)_v \in \Lambda$.
\end{lemma}

\smallskip

Eventually, we say that a group $G$ acts on a $\Lambda$-tree $X$ if any element $g \in G$ defines an
isometry $g : X \rightarrow X$. An action on $X$ is {\em non-trivial} if there is no point in $X$
fixed by all elements of $G$. Note, that every group has a {\em trivial action} on any $\Lambda$-tree,
when all group elements act as identity. An action of $G$ on $X$ is {\em minimal} if $X$ does not
contain a non-trivial $G$-invariant subtree $X_0$.

Let a group $G$ act as isometries on a $\Lambda$-tree $X$. $g \in G$ is called {\em elliptic} if
it has a fixed point. $g \in G$ is called an {\rm inversion} if it does not have a fixed point, but
$g^2$ does. If $g$ is not elliptic and not an inversion then it is called {\em hyperbolic}.

A group $G$  acts {\em freely} and {\em without inversions} on a $\Lambda$-tree $X$ if for all $1
\neq g \in G$, $g$ acts as a hyperbolic isometry. In this case we also say that $G$ is
{\em $\Lambda$-free}.

\subsection{$\Lambda$-free groups}
\label{subs:lambda-free}

Recall that a group $G$ is called {\em $\Lambda$-free} if for all $1 \neq g \in G$, $g$ acts as a
hyperbolic isometry. Here we list some known results about $\Lambda$-free groups for an arbitrary
ordered abelian group $\Lambda$. For all these results the reader can be referred to
\cite{AB, B, Ch1, Martino_Rourke:2005, KMS12}.

\begin{theorem}\label{prop}
\begin{enumerate}
\item[(a)] The class of $\Lambda$-free groups is closed under taking subgroups.
\item[(b)] If $G$ is $\Lambda$-free and $\Lambda$ embeds (as an ordered abelian group) in $\Lambda'$
then $G$ is $\Lambda'$-free.
\item[(c)] Any $\Lambda$-free group is torsion-free.
\item[(d)] $\Lambda$-free groups have the CSA property. That is, every maximal abelian subgroup $A$
is malnormal: $A^g \cap A = 1$ for all $g \notin A$.
\item[(e)] Commutativity is a transitive relation on the set of non-trivial elements of a
$\Lambda$-free group.
\item[(f)] Solvable subgroups of $\Lambda$-free groups are abelian.
\item[(g)] If $G$ is $\Lambda$-free then any abelian subgroup of $G$ can be embedded in $\Lambda$.
\item[(h)] $\Lambda$-free groups cannot contain Baumslag-Solitar subgroups other than $\mathbb{Z}
\times \mathbb{Z}$. That is, no group of the form $\langle a, t \mid t^{-1} a^p t = a^q \rangle$
can be a subgroup of a $\Lambda$-free group unless $p = q = \pm 1$.
\item[(i)] Any two generator subgroup of a $\Lambda$-free group is either free, or free abelian.
\item[(j)] The class of $\Lambda$-free groups is closed under taking free products.
\end{enumerate}
\end{theorem}

The following result was originally proved in \cite{Harrison} in the case of finitely many
factors and $\Lambda = \mathbb{R}$. A proof of the result in the general formulation given below
can be found in \cite[Proposition 5.1.1]{Ch1}.

\begin{theorem}
If $\{G_i \mid i \in I \}$ is a collection of $\Lambda$-free groups then the free product $\ast_{i
\in I} G_i$ is $\Lambda$-free.
\end{theorem}

The following result gives a lot of information about the group structure in the case when $\Lambda
= \mathbb{Z} \times \Lambda_0$ with the left lexicographic order.

\begin{theorem} \cite[Theorem 4.9]{B}
Let a group $G$ act freely and without inversions on a $\Lambda$-tree, where $\Lambda = \mathbb{Z}
\times \Lambda_0$. Then there is a graph of groups $(\Gamma, Y^*)$ such that:
\begin{enumerate}
\item[(1)] $G = \pi_1(\Gamma, Y^*)$,
\item[(2)] for every vertex  $x^* \in Y^*$, a vertex group $\Gamma_{x^*}$ acts freely and without
inversions on a $\Lambda_0$-tree,
\item[(3)] for every edge  $e \in Y^*$ with an endpoint $x^*$ an edge group $\Gamma_e$ is either
maximal abelian subgroup in $\Gamma_{x^*}$ or is trivial and $\Gamma_{x^*}$ is not abelian,
\item[(4)] if $e_1,e_2,e_3 \in Y^*$  are edges with an endpoint $x^*$ then $\Gamma_{e_1},
\Gamma_{e_2}, \Gamma_{e_3}$ are not all conjugate in $\Gamma_{x^*}$.
\end{enumerate}

Conversely, from the existence of a graph $(\Gamma, Y^*)$ satisfying conditions (1)--(4) it follows
that $G$ acts freely and without inversions on a $\mathbb{Z} \times \Lambda_0$-tree in the
following cases: $Y^*$ is a tree, $\Lambda_0 \subset Q$ and either $\Lambda_0 = Q$ or $Y^*$ is finite.
\end{theorem}

\subsection{$\mathbb{R}$-trees}
\label{sec:r-tres}

The case when $\Lambda = \mathbb{R}$ in the theory of groups acting on $\Lambda$-trees is the most well-studied (other than $\Lambda = \mathbb{Z}$, of course). $\mathbb{R}$-trees are usual
metric spaces with nice properties which makes them very attractive from geometric point of view.
The term $\mathbb R$-tree  was introduced by Morgan and Shalen
\cite{MS}  to describe a type of space that was first defined by Tits \cite{Tits:1977}. In the last
three decades $\mathbb R$-trees have played a prominent role in topology, geometry, and geometric group
theory. They are the most simple of geodesic spaces, and yet by Theorem \ref{th:BP} below,
every length space is an orbit space of an $\mathbb R$-tree. Lots of results were obtained in the last two decades about group actions on these objects. The most
celebrated one is Rips' Theorem about free actions and a more general result of  Bestvina and
 Feighn about stable actions on $\mathbb{R}$-trees (see \cite{GLP, BF}).
In particular, the main result of Bestvina and Feighn together with the idea of obtaining a stable
action on an $\mathbb{R}$-tree as a limit of actions on an infinite sequence of $\mathbb{Z}$-trees
gives a very powerful tool in obtaining non-trivial decompositions of groups into fundamental groups of
graphs of groups which is known as {\em Rips machine}. Such decompositions of groups as iterated applications of the operations of free product with amalgamation and HNN extension are called {\em splittings}.

\smallskip

An $\mathbb{R}$-tree $(X,d)$ is a $\Lambda$-metric space which satisfies the axioms (T1) -- (T3)
listed in Subsection \ref{subs:lambda-def} for $\Lambda = \mathbb{R}$ with usual order. Hence, all
the definitions and notions given in Section \ref{sec:lambda} hold for $\mathbb{R}$-trees.

\begin{proposition} \cite[Proposition 2.2.3]{Ch1}
\label{pr:eq_def}
Let $(X,d)$ be an $\mathbb{R}$-metric space. Then the following are equivalent:
\begin{enumerate}
\item $(X,d)$ is an $\mathbb{R}$-tree,
\item given two points of $X$, there is a unique segment (with more than one point) having them as endpoints, 
\item $(X,d)$ is geodesic and it contains no subspace homeomorphic to the circle.
\end{enumerate}
\end{proposition}

\begin{example}
\label{examp:r_tree1}
Let $Y = \mathbb{R}^2$ be the plane, but with metric $p$ defined by
\[
p((x_1,y_1),(x_2,y_2)) = \left\{ \begin{array}{ll}
\mbox{$|y_1| + |y_2| + |x_1 - x_2|$}  & \mbox{if $x_1 \neq x_2$} \\
\mbox{$|y_1 - y_2|$ } & \mbox{if $x_1 = x_2$}
\end{array}
\right.
\]
That is, to measure the distance between two points not on the same vertical line, we take their
projections onto the horizontal axis, and add their distances to these projections and the distance
between the projections (distance in the usual Euclidean sense).
\end{example}

\begin{example} \cite[Proposition 2.2.5]{Ch1}
\label{examp:r_tree2}
Given a simplicial tree $\Gamma$, one can construct its realization $real(\Gamma)$ by identifying
each non-oriented edge of $\Gamma$ with the unit interval. The metric on $real(\Gamma)$ is induced
from $\Gamma$.
\end{example}

\begin{example}
\label{exam:asympt_cone}
Let $G$ be a $\delta$-hyperbolic group. Then its Cayley graph with respect to any finite generating
set $S$ is a $\delta$-hyperbolic metric space $(X, d)$ (where $d$ is a word metric) on which $G$ acts
by isometries. Now, the asymptotic cone $Cone_\omega(X)$ of $G$ is a real tree (see \cite{Gr, Drutu_Sapir:2005, WW}) on which $G$ acts by isometries. 
\end{example}

An $\mathbb{R}$-tree is called {\em polyhedral} if the set of all branch points and endpoints is
closed and discrete. Polyhedral $\mathbb{R}$-trees have strong connection with simplicial trees as
shown below.

\begin{theorem} \cite[Theorem 2.2.10]{Ch1}
\label{th:polyhedral}
An $\mathbb{R}$-tree $(X,d)$ is polyhedral if and only if it is homeomorphic to $real(\Gamma)$ (with
metric topology) for some simplicial tree $\Gamma$.
\end{theorem}

Now we briefly recall some known results related to group actions on $\mathbb{R}$-trees. The first
result shows that an action on a $\Lambda$-tree always implies an action on an $\mathbb{R}$-tree.

\begin{theorem} \cite[Theorem 4.1.2]{Ch1}
\label{th:lambda_r}
If a finitely generated group $G$ has a non-trivial action on a $\Lambda$-tree for some ordered
abelian group $\Lambda$ then it has a non-trivial action on some $\mathbb{R}$-tree.
\end{theorem}

Observe that in general nice properties of the action on a $\Lambda$-tree are not preserved when
passing to the corresponding action on an $\mathbb{R}$-tree above.

The next result was one of the first in the theory of group actions on $\mathbb{R}$-trees. Later it was
generalized to the case of an arbitrary $\Lambda$, see Theorem \ref{prop}.

\begin{theorem} \cite{Harrison}
Let $G$ be a group acting freely and  without inversions on an $\mathbb{R}$-tree $X$, and suppose
$g, h \in G\ \backslash \ \{1\}$. Then $\langle g, h \rangle$ is either free of rank two or abelian.
\end{theorem}

Let $(X,d$) be a metric space. Then $d$ is said to be a {\em length metric} if the distance between every pair of points $x,y\in X$ is equal to the infimum of the length of rectifiable curves joining them. (If there are no such curves then $d(x,y)=\infty.$) If $d$ is a length metric then $(X,d)$ is called a {\em length space}. 

\begin{theorem} \label{th:BP}\cite{Berestovski_Paoul} Every length space (resp. complete length space) $(X,d$) is the metric quotient of a
(resp. complete) $\mathbb R$-tree $(X, d)$ via the free isometric action of a locally free subgroup $\Gamma (X)$ of
the isometry group $Isom(X)$. 
\end{theorem}

It is not hard to define an action of a free abelian group on an $\mathbb{R}$-tree.

\begin{example}
\label{exam:ab_act}
Let $A = \langle a,b \rangle$ be a free abelian group. Define an action of $A$ on $\mathbb{R}$
(which is an $\mathbb{R}$-tree) by embedding $A$ into $Isom(\mathbb{R})$ as follows
$$a \rightarrow t_1,\ \ b \rightarrow t_{\sqrt{2}},$$
where $t_\alpha(x) = x + \alpha$ is a translation. It is easy to see that
$$a^n b^m \rightarrow t_{n+m\sqrt{2}},$$
and since $1$ and $\sqrt{2}$ are rationally independent it follows that the action is free.
\end{example}

The following result was very important in the direction of classifying finitely generated
$\mathbb{R}$-free groups.

\begin{theorem} \cite{Morgan_Shalen:1991}
\label{th:Morgan_Shalen}
The fundamental group of a closed surface is $\mathbb{R}$-free, except for the non-orientable
surfaces of genus $1,2$ and $3$.
\end{theorem}

Then, in 1991 E. Rips completely classified finitely generated $\mathbb{R}$-free groups. The ideas
outlined by Rips were further developed by Gaboriau, Levitt and Paulin who gave a complete proof
of this classification in \cite{GLP}.

\begin{theorem}[Rips' Theorem]
\label{th:Rips_0}
Let $G$ be a finitely generated group acting freely and without inversions on an $\mathbb{R}$-tree.
Then $G$ can be written as a free product $G = G_1 \ast \cdots \ast G_n$ for some integer $n
\geqslant 1$, where each $G_i$ is either a finitely generated free abelian group, or the
fundamental group of a closed surface.
\end{theorem}

It is worth mentioning that
there are examples by Dunwoody [17] and Zastrow [38] of infinitely generated groups that are
not free products of fundamental groups of closed surfaces and abelian groups, but which act
freely on an R-tree. Zastrow's group G contains one of the two Dunwoody groups as a subgroup.
The other group is a Kurosh group.   

Berestovskii and Plaut proved the following result.
We first introduce necessary notation. If $(X,d)$ is the length space, then ${\mathbb R}$-tree $\overline X$ is defined as the space of based ``non-backtracking'' rectifiable paths
in $X$, where the distance between two paths is the sum of their lengths from the first
bifurcation point to their endpoints. The group $\Gamma (X) \subset \overline X$ is the subset of loops with
a natural group structure and the quotient mapping $\phi : \overline X\rightarrow X$ is the end-point map.
We will refer to $\overline X$ as the covering $\mathbb R$-tree of $X$. 

\begin{theorem} Let $X$ be $S_c$ (Sierpinski carpet), $M$ a complete Riemannian manifold $M_n$ of dimension $n=  2$, or the
Hawaiian earring $H$ with any compatible length metric $d$. Then $\Gamma  (X)$ is an infinitely generated,
locally free group that is not free and not a free product of surface groups and abelian groups, but
acts freely on the $\mathbb R$-tree $\overline X$. Moreover, the $\mathbb R$-tree $\overline X$ is a minimal invariant subtree with respect to  this action.
\end{theorem}

\subsection{Rips-Bestvina-Feighn machine}
\label{sec:rbf}

Suppose $G$ is a finitely presented group acting isometrically on an $\R$-tree $\Gamma$. We assume
the action to be non-trivial and minimal. Since $G$ is finitely presented there is a finite simplicial
complex $K$ of dimension at most $2$ such that $\pi_1(K) \simeq G$. Moreover, one can assume that
$K$ is a {\em band complex} with underlying {\em union of bands} which is a finite simplicial $\R$-tree $X$ with finitely many {\em bands}
of the type $[0,1] \times \alpha$, where $\alpha$ is an arc of the real line, glued to $X$ so that
$\{0\} \times \alpha$ and $\{1\} \times \alpha$ are identified with sub-arcs of edges of $X$.
Following \cite{BF} (the construction originally appears in \cite{MS})
one can construct a transversely measured lamination $L$ on $K$ and an equivariant map $\phi : \widetilde{K} \to
\Gamma$, where $\widetilde{K}$ is the universal cover of $K$, which sends leaves of the induced
lamination on $\widetilde{K}$ to points in $\Gamma$. The  complex $K$ together with the
lamination $L$ is called a band complex with $\widetilde{K}$ resolving the action of $G$ on
$\Gamma$.

Now, {\em Rips-Bestvina-Feighn machine} is a procedure which given a band complex $K$, transforms
it into another band complex $K'$ (we still have $\pi_1(K') \simeq G$), whose lamination splits into a
disjoint union of finitely many sub-laminations of several types - simplicial, surface, toral, thin
- and these sub-laminations induce a splitting of $K'$ into sub-complexes containing them. $K'$ can
be thought of as the ``normal form'' of the band complex $K$. Analyzing the structure of $K'$ and
its sub-complexes one can obtain some information about the structure of the group $G$.

In particular, in the case when the original action of $G$ on $\Gamma$ is {\em stable} one can
obtain a splitting of $G$. Recall that a non-degenerate (that is, containing more than one point)
subtree $S$ of $\Gamma$ is {\em stable} if for every non-degenerate subtree $S'$ of $S$, we have
$Fix(S') = Fix(S)$ (here, $Fix(I) \leqslant G$ consists of all elements which fix $I$ point-wise). The action
of $G$ on $\Gamma$ is {\em stable} if every non-degenerate subtree of $T$ contains a stable subtree.
We say  that a group $G$ {\em splits over} a subgroup $E$ is $G$ is either an amalgamated product or an HNN extension over $E$ (with the edge group $E$). 

\begin{theorem} \cite[Theorem 9.5]{BF}
\label{th:Best_Feighn}
Let $G$ be a finitely presented group with a nontrivial, stable, and minimal action on an
$\R$-tree $\Gamma$. Then either
\begin{enumerate}
\item[(1)] $G$ splits over an extension $E$-by-cyclic, where $E$ fixes a segment of $\Gamma$, or
\item[(2)] $\Gamma$ is a line. In this case, $G$ splits over an extension of the kernel of the
action by a finitely generated free abelian group.
\end{enumerate}
\end{theorem}

The key ingredient of the Rips-Bestvina-Feighn machine is a set of particular operations, called
{\em moves}, on band complexes applied in a certain order. These operations originate from
the work of Makanin \cite{Mak82} and Razborov \cite{Razborov1} that ideas of
Rips are built upon.

Observe that the group $G$ in Theorem \ref{th:Best_Feighn} must be finitely presented. To obtain a
similar result about finitely generated groups acting on $\R$-trees one has to further restrict
the action. An action of a group $G$ on an $\R$-tree $\Gamma$ satisfies the {\em ascending
chain condition} if for every decreasing sequence
$$I_1 \supset I_2 \supset \cdots \supset I_n \supset \cdots$$
of arcs in $\Gamma$ which converge into a single point, the corresponding sequence
$$Fix(I_1) \subset Fix(I_2) \subset \cdots \subset Fix(I_n) \subset \cdots$$
stabilizes.

\begin{theorem} \cite{Guirardel:2008}
\label{th:Guirardel_0}
Let $G$ be a finitely generated group with a nontrivial minimal action on an $\R$-tree $\Gamma$.
If
\begin{enumerate}
\item[(1)] $\Gamma$ satisfies the ascending chain condition,
\item[(2)] for any unstable arc $J$ of $\Gamma$,
\begin{enumerate}
\item[(a)] $Fix(J)$ is finitely generated,
\item[(b)] $Fix(J)$ is not a proper subgroup of any conjugate of itself, that is, if $Fix(J)^g
\subset Fix(J)$ for some $g \in G$ then $Fix(J)^g = Fix(J)$.
\end{enumerate}
\end{enumerate}
Then either
\begin{enumerate}
\item[(1)] $G$ splits over a subgroup $H$ which is an extension of the stabilizer of an arc of
$\Gamma$ by a cyclic group, or
\item[(2)] $\Gamma$ is a line.
\end{enumerate}
\end{theorem}

Now, we will discuss some applications of the above results which are based on the
construction outlined in \cite{Bestvina:1988} and \cite{Paulin:1988} making possible to obtain
isometric group actions on $\R$-trees as Gromov-Hausdorff limits of actions on hyperbolic spaces.
All the details can be found in \cite{Best}.

Let $(X, d_X)$ be a metric space equipped with an isometric action of a group $G$ which can be viewed as
a homomorphism $\rho: G \to Isom(X)$. Assume that $X$ contains a point $\varepsilon$ which is not
fixed by $G$. In this case, we call the triple $(X, \varepsilon, \rho)$ a {\em based $G$-space}.

Observe that given a based $G$-space $(X, \varepsilon, \rho)$ one can define a pseudometric $d =
d_{(X, \varepsilon, \rho)}$ on $G$ as follows
$$d(g,h) = d_X(\rho(g) \cdot \varepsilon, \rho(h) \cdot \varepsilon).$$
Now, the set ${\cal D}$ of all non-trivial pseudometrics on $G$ equipped with compact open topology and then  taken up to rescaling by positive
reals (projectivized), forms a topological space and we say that a sequence $(X_i, \varepsilon_i, \rho_i),\ i \in \N$
of based $G$-spaces {\em converges} to the based $G$-space $(X, \varepsilon, \rho)$ and write
$$\lim_{i \to \infty} (X_i, \varepsilon_i, \rho_i) = (X, \varepsilon, \rho)$$
if we have $[d_{(X_i, \varepsilon_i, \rho_i)}] \to [d_{(X, \varepsilon, \rho)}]$ for the projectivized equivariant pseudometrics in ${\cal D}.$ Now, the following
result is the main tool in obtaining isometric group actions on $\R$-trees from actions on
Gromov-hyperbolic spaces.

\begin{theorem} \cite[Theorem 3.3]{Best}
\label{th:Paulin}
Let $(X_i, \varepsilon_i, \rho_i), i \in \N$ be a convergent sequence of based $G$-spaces. Assume
that
\begin{enumerate}
\item[(1)] there exists $\delta > 0$ such that every $X_i$ is $\delta$-hyperbolic,
\item[(2)] there exists $g \in G$ such that the sequence $d_{X_i}(\varepsilon_i, \rho_i(g) \cdot
\varepsilon_i)$ is unbounded.
\end{enumerate}
Then there is a based $G$-space $(\Gamma, \varepsilon)$ which is an $\R$-tree and an isometric
action $\rho : G \to Isom(\Gamma)$ such that $(X_i, \varepsilon_i, \rho_i) \to (\Gamma, \varepsilon,
\rho)$.
\end{theorem}

In fact, the above theorem does not guarantee that the limiting action of $G$ on $\Gamma$ has no
global fixed points. But in the case when $G$ is finitely generated and each $X_i$ is proper (closed
metric balls are compact), it is possible to choose base-points in $\varepsilon_i \in X_i$ to make
the action of $G$ on $\Gamma$ non-trivial (see \cite[Proposition 3.8, Theorem 3.9]{Best}).
Moreover, one can retrieve some information about stabilizers of arcs in $\Gamma$ (see
\cite[Proposition 3.10]{Best}).

Note that Theorem \ref{th:Paulin} can also be interpreted in terms of asymptotic cones (see
\cite{Drutu_Sapir:2005, Drutu_Sapir:2008} for details).

\smallskip

The power of Theorem \ref{th:Paulin} becomes obvious in particular when a finitely generated group
$G$ has infinitely many pairwise non-conjugate homomorphisms $\phi_i : G \to H$ into a word-hyperbolic
group $H$. In this case, each $\phi_i$ defines an action of $G$ on the Cayley graph $X$ of $H$ with
respect to some finite generating set. Now, one can define $X_i$ to be $X$ with a word metric rescaled so that
the sequence of $(X_i, \varepsilon_i, \rho_i), i \in \N$ satisfies the requirements of Theorem \ref{th:Paulin} and
thus obtain a non-trivial isometric action of $G$ on an $\R$-tree. Many results about word-hyperbolic
groups were obtained according to this scheme, for example, the following classical result.

\begin{theorem} \cite{Paulin:1991}
\label{th:Paulin_2}
Let $G$ be a word-hyperbolic group such that the group of its outer automorphisms $Out(G)$ is infinite.
Then $G$ splits over a virtually cyclic group.
\end{theorem}

Combined with the {\em shortening argument} due to Rips and Sela \cite{RipsSela:1994} this scheme
gives many other results about word-hyperbolic groups, for example, the theorems below.

\begin{theorem} \cite{RipsSela:1994}
\label{th:RipsSela}
Let $G$ be a torsion-free freely indecomposable word-hyperbolic group. Then the internal automorphism
group (that consists of automorphisms obtained by compositions of Dehn twists and inner automorphisms)  of $G$ has finite index in $Aut(G)$.
\end{theorem}

\begin{theorem} \cite{Gromov:1987, Sela:1997}
\label{th:GromovSela}
Let $G$ be a finitely presented torsion-free freely indecomposable group and let $H$ be a
word-hyperbolic group. Then there are only finitely many conjugacy classes of subgroups of $G$
isomorphic to $H$.
\end{theorem}

There are similar recent results for relatively hyperbolic groups \cite{GuiL}.

For more detailed account of applications of the Rips-Bestvina-Feighn machine please refer to
\cite{Best}.

\subsection{Lyndon length functions}
\label{sec:length_func}

In 1963 Lyndon (see \cite{Lyndon:1963}) introduced a notion of {\em length function on a group}
in an attempt to axiomatize cancelation arguments in free groups as well as free products with
amalgamation and HNN extensions, and to generalize them to a wider class of groups. The main idea
was to measure the amount of cancellation in passing to the reduced form of a product of reduced
words in a free group and free constructions, and it turned out that the cancelation process could
be described by rather simple axioms. 
The idea of using length functions became quite popular (see, for example, \cite{Harrison,
Chiswell:1976, Hoare1}), and then it turned out that the language of length functions described
the same class of groups as the language of actions on trees. 
Below we give the axioms of (Lyndon) length function and recall the main results in this field.

\smallskip

Let $G$ be a group and $\Lambda$ be an ordered abelian group. Then a function $l: G \rightarrow
\Lambda$ is called a {\em (Lyndon) length function} on $G$ if the following conditions hold:
\begin{enumerate}
\item [(L1)] $\forall\ g \in G:\ l(g) \geqslant 0$ and $l(1) = 0$,
\item [(L2)] $\forall\ g \in G:\ l(g) = l(g^{-1})$,
\item [(L3)] $\forall\ f, g, h \in G:\ c(f,g) > c(f,h)$ implies $c(f,h) = c(g,h)$,

\noindent where $c(f,g) = \frac{1}{2}(l(f) + l(g) - l(f^{-1}g))$.
\end{enumerate}

Observe that in general $c(f,g) \notin \Lambda$, but $c(f,g) \in \Lambda_{\mathbb{Q}} = \Lambda
\otimes_{\mathbb{Z}} \mathbb{Q}$, where $\mathbb{Q}$ is the additive group of rational numbers,
so, in the axiom (L3) we view $\Lambda$ as a subgroup of $\Lambda_{\mathbb{Q}}$. But in some cases
the requirement $c(f,g) \in \Lambda$ is crucial so we state it as a separate axiom \cite{KMS12}
\begin{enumerate}
\item [(L4)] $\forall\ f, g \in G:\ c(f,g) \in \Lambda.$
\end{enumerate}

It is not difficult to derive the following two properties of Lyndon length functions from the
axioms (L1) -- (L3):
\begin{itemize}
\item $\forall\ f, g \in G:\ l(f g) \leqslant l(f) + l(g)$,
\item $\forall\ f, g \in G:\ 0 \leqslant c(f,g) \leqslant \min\{l(f),l(g)\}$.
\end{itemize}

The following examples motivated the whole theory of groups with length functions.

\begin{example}
\label{exam:free_gr}
Given a free group $F(X)$ on the set $X$ one can define a (Lyndon) length function on $F$ as follows
$$w(X) \rightarrow |w(X)|,$$
where $| \cdot |$ is the length of the reduced word in $X \cup X^{\pm 1}$ representing $w$.
\end{example}

\begin{example}
\label{exam:free_prod}
Given two groups $G_1$ and $G_2$ with length functions $L_1 : G_1 \rightarrow \Lambda$ and $L_2 :
G_2 \rightarrow \Lambda$ for some ordered abelian group $\Lambda$ one can construct a length
function on $G_1 \ast G_2$ as follows (see \cite[Proposition 5.1.1]{Ch1}). For any $g \in G_1
\ast G_2$ such that
$$g = f_1 g_1 \cdots f_k g_k f_{k+1},$$
where $f_i \in G_1,\ i \in [1, k+1],\ f_i \neq 1,\ i \in [2, k]$ and $1 \neq g_i \in G_2,\ i \in
[1, k]$, define  $L(1)=0$ and if $g\neq 1$ then
$$L(g) = \sum_{i=1}^{k+1} L_1(f_i) + \sum_{j=1}^k L_2(g_j) \in \Lambda.$$
\end{example}

A length function $l: G \rightarrow \Lambda$ is called {\em free}, if it satisfies
\begin{enumerate}
\item [(L5)] $l(g^2) > l(g)$ for all non-trivial $g\in G.$
\end{enumerate}

Obviously, the $\mathbb{Z}$-valued length function constructed in Example \ref{exam:free_gr} is
free. The converse is shown below (see also \cite{Hoare1} for another proof of this result).

\begin{theorem} \cite{Lyndon:1963}
\label{th:lyndon}
Any group $G$ with a length function $L : G \rightarrow \mathbb{Z}$ can be embedded into a free
group $F$ of finite rank whose natural length function extends $L$.
\end{theorem}

\begin{example}
\label{exam:free_prod_2}
Given two groups $G_1$ and $G_2$ with free length functions $L_1 : G_1 \rightarrow \Lambda$ and $L_2 :
G_2 \rightarrow \Lambda$ for some ordered abelian group $\Lambda$, the length function on $G_1 \ast
G_2$ constructed in Example \ref{exam:free_prod} is free.
\end{example}

Observe that if a group $G$ acts on a $\Lambda$-tree $(X,d)$ then we can fix a point $x \in X$ and
consider a function $l_x : G \to \Lambda$ defined as $l_x(g) = d(x, g x)$. Such a function $l_x$ on
$G$ we call a {\em length function based at $x$}. It is easy to check that $l_x$ satisfies all the
axioms (L1) -- (L4) of Lyndon length function. Now if $\| \cdot \|$ is the translation length function
associated with the action of $G$ on $(X,d)$ (for $g\in G$, $||g|| = Inf \{d(p,gp), p\in X\}$), 
then the following properties show the connection between
$l_x$ and $\| \cdot \|$.
\begin{enumerate}
\item[(i)] $l_x(g) = \| g \| + 2 d(x, A_g)$ (where $A_g$ is the axis of $g$), if $g$ is not an inversion.
\item[(ii)] $\| g \| = \max\{0, l_x(g^2) - l_x(g)\}$.
\end{enumerate}
Here, it should be noted that for points $x \notin A_g$, there is a unique closest point of $A_g$
to $x$. The distance between these points is the one referred to in (i). While $A_g = A_{g^n}$ for
all $n \neq 0$ in the case where $g$ is hyperbolic, if $g$ fixes a point, it is possible that
$A_g \subset A_{g^2}$. We may have $l_x(g^2) - l_x(g) < 0$ in this case. Free actions are
characterized, in the language of length functions, by the facts (a) $\| g \| > 0$ for all $g
\neq 1$, and (b) $l_x(g^2) > l_x(g)$ for all $g \neq 1$. The latter follows from the fact that
$\| g^n \| = n \| g \|$ for all $g$. We note that there are properties for the translation length function
which were shown to essentially characterize actions on $\Lambda$-trees, up to equivariant isometry,
by Parry, \cite{Parry:1991}.

The following theorem is one of the most important results in the theory of length functions.

\begin{theorem} \cite{Chiswell:1976}
Let $G$ be a group and $l: G \rightarrow \Lambda$ a Lyndon length function satisfying 
condition (L4).
Then there are a $\Lambda$-tree $(X, d)$, an action of $G$ on $X$, and a point $x \in X$ such that
$l = l_x$.
\end{theorem}

The proof is constructive, that is, one can define a $\Lambda$-metric space out of $G$ and $l$, and
then prove that this space is in fact a $\Lambda$-tree on which $G$ acts by isometries (see
\cite[Subsection 2.4]{Ch1} for details).

A length function $l: G \rightarrow \Lambda$ is  called {\it regular} if it satisfies the {\em
regularity} axiom:
\begin{enumerate}
\item[(L6)] $\forall\ g, f \in G,\ \exists\ u, g_1, f_1 \in G:$
$$g = u \circ g_1 \ \& \  f = u \circ f_1 \ \& \ l(u) = c(g,f).$$
\end{enumerate}

Observe that a regular length function need not be free, conversely freeness does not imply
regularity.  

\subsection{Finitely generated $\mathbb{R}^n$-free groups}
\label{se:R^n-free}

Guirardel proved the following result that describes the structure of finitely generated
$\mathbb{R}^n$-free groups, which is reminiscent of the Bass' structural theorem for $\Z^n$-free
groups. This is not by chance, since every $\Z^n$-free group is also $\R^n$-free, and ordered
abelian groups $\Z^n$ and $\R^n$ have a similar convex subgroup structure.  However, it is worth to
point out that the original Bass argument for $\Lambda = \Z \oplus \Lambda_0$ does not work in the
case of $\Lambda = \R \oplus \Lambda_0$.

\begin{theorem} \cite{G}
\label{th:Guirardel}
Let $G$ be a finitely generated, freely indecomposable $\mathbb{R}^n$-free group. Then $G$ can be
represented as the fundamental group of a finite graph of groups, where edge groups are cyclic and
each vertex group is a finitely generated $\mathbb{R}^{n-1}$-free.
\end{theorem}

In fact, there is a more detailed version  of this result,  Theorem 7.2 in \cite{G},
which is rather technical, but gives more for applications. Observe also that neither Theorem
\ref{th:Guirardel} nor the more detailed version of it,  does not ``characterize" finitely generated
$\R^n$-free groups, i.e. the converse of the theorem does not hold. Nevertheless, the result is
very powerful and gives several important corollaries.

\begin{corollary} \cite{G}
\label{co:Guirardel_1}
Every finitely generated $\mathbb{R}^n$-free group is finitely presented.
\end{corollary}
This comes from Theorem \ref{th:Guirardel} and elementary properties of free constructions by
induction on $n$.

Theorem \ref{th:Guirardel} and the Combination Theorem for
relatively hyperbolic groups proved by  Dahmani in  \cite{Dahmani:2003} imply the following.

\begin{corollary}
\label{co:Guirardel_2}
Every finitely generated $\mathbb{R}^n$-free group is hyperbolic relative to its non-cyclic
abelian subgroups.
\end{corollary}

A lot is known about groups which are hyperbolic relative to its maximal abelian subgroups ({\em toral
relatively hyperbolic groups}), so all of this applies to $\R^n$-free groups. We do not mention any of
these results here, because we discuss their much more general  versions in the next section in the
context of $\Lambda$-free groups for arbitrary $\Lambda$.

\section{Finitely presented $\Lambda$-free groups}\label{S3}
In  \cite{KMS12}  the following main
problem of the Alperin-Bass program was  solved  for finitely presented groups.

\medskip

{\bf Problem.}
{\it Describe finitely presented (finitely generated) $\Lambda$-free groups for an arbitrary ordered
abelian group $\Lambda$.}

The structure of finitely presented $\Lambda$-free groups is going to be described in Section \ref{se:fp}.

\subsection{Regular actions}
\label{subse:regact}

In this section we give a geometric characterization of group actions that come from regular
length functions.

\smallskip

Let $G$ act on a $\Lambda$-tree $\Gamma$. The action is {\em regular with respect to $x \in \Gamma$}
if for any $g,h \in G$ there exists $f \in G$ such that $[x, f x] = [x, g x] \cap [x, h x]$.

\smallskip

The next lemma shows that regular actions exactly correspond to regular length functions (hence the
term).
\begin{lemma} \cite{KMRS1}
Let $G$ act on a $\Lambda$-tree $\Gamma$. Then the action of $G$ is regular with respect to $x \in
\Gamma$ if and only if the length function $\ell_x: G \rightarrow \Lambda$ based at $x$ is regular.

If the action of $G$ is regular with respect to $x \in
\Gamma$ then it is regular with respect to any $y \in G x$.
\end{lemma}

\begin{lemma} \cite{KMRS1}
Let $G$ act freely on a $\Lambda$-tree $\Gamma$ so that all branch points of $\Gamma$ are
$G$-equivalent. Then the action of $G$ is regular with respect to any branch point in $\Gamma$.
\end{lemma}
\begin{proof}
Let $x$ be a branch point in $\Gamma$ and $g, h \in G$. If $g = h$ then $[x, g x] \cap [x, h x] =
[x, g x]$ and $g$ is the required element. Suppose $g \neq h$. Since the action is free then $g x
\neq h x$ and we consider the tripod formed by $x, g x, h x$. Hence, the center of the tripod $y$ is a
branch point in $\Gamma$ and by the assumption there exists $f \in G$ such that $y = f x$.
\end{proof}


\begin{example}
\label{example_non_reg}
Let $\Gamma'$ be the Cayley graph of a free group $F(x,y)$ with the base-point $\varepsilon$. Let
$\Gamma$ be obtained from $\Gamma'$ by adding an edge labeled by $z \neq x^{\pm 1}, y^{\pm 1}$ at
every vertex of $\Gamma'$. $F(x,y)$ has a natural action on $\Gamma'$ which we can extend to the
action on $\Gamma$. The edge at $\varepsilon$ labeled by $z$ has an endpoint not equal to
$\varepsilon$ and we denote it by $\varepsilon'$. Observe that the action of $F(x,y)$ on $\Gamma$
is regular with respect to $\varepsilon$ but is not regular with respect to $\varepsilon'$.
\end{example}

\subsection{Structure of finitely presented $\Lambda$-free groups}\label{se:fp}
A group $G$ is called  a regular $\Lambda$-free group if it acts freely and regularly on a $\Lambda$-tree. We proved the following results. 
\begin{theorem}\cite{KMS12} 
\label{th:main1}
Any finitely presented regular $\Lambda$-free group $G$  can be represented as a union of a finite
series of groups
$$G_1 < G_2 < \cdots < G_n = G,$$
where
\begin{enumerate}
\item $G_1$ is a free group,
\item $G_{i+1}$ is obtained from $G_i$ by finitely many HNN-extensions in which associated subgroups
are maximal abelian, finitely generated, and the associated isomorphisms preserve the length induced from $G_i$.
\end{enumerate}
\end{theorem}

\begin{theorem}\cite{KMS12} 
\label{th:main4}
Any finitely presented $\Lambda$-free group can be isometrically embedded into a finitely presented
regular $\Lambda$-free group.
\end{theorem}

\begin{theorem}\cite{KMS12} 
\label{th:main3}
Any finitely presented  $\Lambda$-free group $G$ is $\mathbb{R}^n$-free for an appropriate
$n \in \mathbb{N}$, where $\mathbb{R}^n$ is ordered lexicographically.
\end{theorem}

In his book \cite{Ch1} Chiswell (see also \cite{Remeslennikov:1989})  asked the following
very important question (Question 1, p. 250): If $G$ is a finitely generated $\Lambda$-free group,
is $G$ $\Lambda_0$-free for some finitely generated abelian ordered group $\Lambda_0$? The
following result answers this question in the affirmative in the strongest form. It comes from the
proof of Theorem \ref{th:main3} (not the statement of the theorem itself).

\begin{theorem}\cite{KMS12} 
\label{co:main1}
Let $G$ be a finitely presented group with a free Lyndon length function $l : G \to \Lambda$. Then
the subgroup $\Lambda_0$ generated by $l(G)$ in $\Lambda$ is finitely generated.
\end{theorem}

The following result  follows directly from Theorem \ref{th:main1} and Theorem \ref{th:main4} by a
simple application of Bass-Serre Theory.

\begin{theorem}\cite{KMS12} 
\label{co:main5}
Any finitely presented $\Lambda$-free group $G$ can be obtained from free groups by a finite sequence of amalgamated
free products and HNN extensions along maximal abelian subgroups, which are free abelian groups of finite rank.
\end{theorem}

The following result concerns  abelian subgroups of $\Lambda$-free groups. For $\Lambda = {\mathbb Z}^n$ it follows from the main structural result for ${\mathbb Z}^n$-free groups and \cite{[66]}, for $\Lambda = {\mathbb R}^n$ it was proved in \cite{G}. The statement 1) below answers  Question 2 (page 250) from \cite{Ch1} in the affirmative for finitely presented $\Lambda$-free groups.

\begin{theorem}\cite{KMS12} 
\label{co:main1b}
Let $G$ be a finitely presented $\Lambda$-free group. Then:
\begin{itemize}
\item [1)] every abelian subgroup of $G$ is a free abelian group of finite rank, which is uniformly bounded from above by the rank of the abelianization of $G$.
\item [2)] $G$ has only finitely many conjugacy classes of
maximal non-cyclic abelian subgroups,
\item [3)] $G$  has a finite classifying space and the cohomological dimension of $G$  is $r$ where $r$  is the maximal rank of an abelian subgroup of $G$, except if $r=1$ when the cohomological dimension of $G$ is 1 or 2.
\end{itemize}
\end{theorem}

\begin{theorem}\cite{KMS12} 
\label{co:Lambda_Guirardel_2}
Every finitely presented  $\Lambda$-free group is hyperbolic relative to its non-cyclic abelian
subgroups.
\end{theorem}

This follows from  the structural Theorem \ref{th:main1} and  the Combination Theorem for relatively
hyperbolic groups \cite{Dahmani:2003}.

The following results answers affirmatively  the strongest form of the Problem (GO3) from the Magnus
list of open problems \cite{BaumMyasShpil:2002}, in the case of finitely presented groups.

\begin{corollary}
\label{co:Lambda_Guirardel_3}
Every finitely presented  $\Lambda$-free group  is biautomatic.
\end{corollary}
\begin{proof} This follows from Theorem \ref{co:Lambda_Guirardel_2} and Rebbecchi's result
\cite{Rebeca}.
\end{proof}
\begin{definition} A {\em hierarchy} for a group $G$ is a way to construct it  starting with trivial groups by repeatedly taking amalgamated products $A*_CB$ and $HNN$ extensions $A*_{C^t=D}$ whose vertex groups have shorter length hierarchies. The hierarchy is {\em quasi convex} if the amalgamated subgroup $C$ is a finitely generated subgroup that embeds by a quasi-isometric embedding, and if
 $C$ is malnormal in $A*_cB$ or $A*_{C^t=D}$ .
\end{definition}
\begin{theorem}
\label{th:Lambda_quasi-convex_hierarchy}
Every finitely presented $\Lambda$-free group $G$ has a quasi-convex hierarchy with abelian edge groups.
\end{theorem}
\begin{theorem}\cite{wise} Suppose $G$ is toral relatively hyperbolic and has a malnormal quasi convex hierarchy. Then $G$ is virtually special (therefore has a finite index subgroup that is a subgroup of a right angled Artin group (RAAG)). \end{theorem}

As a corollary one gets the following result.
\begin{corollary}
\label{co:Lambda_quasi-convex_hierarchy_1}
Every finitely presented  $\Lambda$-free group $G$ is locally undistorted, that is, every finitely generated subgroup of $G$ is quasi-isometrically embedded into $G$.
\end{corollary}

\begin{corollary}
\label{co:Lambda_special}
Every finitely presented $\Lambda$-free group $G$ is  virtually special, that is, some subgroup of
finite index in $G$ embeds into a right-angled Artin group.
\end{corollary}

The following result answers in the affirmative to Question 3
(page 250) from \cite{Ch1} in the case of finitely presented groups. 

\begin{theorem}\cite{KMS12} 
\label{co:Lambda_special_0}
Every finitely presented $\Lambda$-free group is right orderable and virtually orderable.
\end{theorem}

Since right-angled Artin groups are linear and
the class of linear groups is closed under finite extension we get the following

\begin{theorem}
\label{co:Lambda_quasi-convex_hierarchy_2}
Every finitely presented  $\Lambda$-free group is  linear and, therefore, residually finite, equationally noetherian. It also has decidable word, conjugacy, subgroup membership, and diophantine problems.
\end{theorem}

Indeed, decidability of equations follows from \cite{Dahmani:2003}. Results of Dahmani and Groves \cite{DG}
imply the following two corollaries.

\begin{corollary}
Let $G$ be a finitely presented  $\Lambda$-free group. Then:
\begin{itemize}
\item $G$ has a non-trivial abelian splitting and one can find such a splitting effectively,
\item $G$ has a non-trivial abelian JSJ-decomposition and one can find such a decomposition
effectively.
\end{itemize}
\end{corollary}

\begin{corollary}
The isomorphism problem is decidable in the class of finitely presented   $\Lambda$-free groups. 
\end{corollary}

\subsection{Limit groups}
\label{sec:limit_gps}

Limit groups (or finitely generated fully residually free groups) play an important part in modern group theory. They appear in many different situations:
in combinatorial group theory as groups discriminated by $G$ ($\omega$-residually $G$-groups or
fully residually $G$-groups) \cite{Baumslag:1962, Baumslag:1967, Myasnikov_Remeslennikov:2000,
BMR1, Baumslag_Miasnikov_Remeslennikov:2002}, in the algebraic
geometry over groups as the coordinate groups of irreducible varieties over $G$
\cite{BMR1, KMNull,
KMIrc, Imp, Sela1}, groups universally
equivalent to $G$ \cite{Remeslennikov:1989, GaglioneSpellman:1993, Myasnikov_Remeslennikov:2000}, limit groups of
$G$ in the Gromov-Hausdorff topology \cite{Champetier_Guirardel:2005}, in the theory of equations
in groups \cite{Lyndon2, Razborov1, Razborov:1995, KMNull,
KMIrc, Imp, Groves:2005(2)}, in group actions
\cite{BF, GLP, MRS2, Groves:2005(1),
G},  in the solutions of Tarski problems \cite{KM4, Sela6},
etc. Their numerous characterizations  connect  group
theory, topology and logic.

Recall, that   a group $G$ is called {\em fully residually free} if for any non-trivial $g_1, \ldots,
g_n \in G$ there exists a homomorphism $\phi$ of $G$ into a free group such that $\phi(g_1), \ldots,
\phi(g_n)$ are non-trivial.

It is a crucial result  that every limit group admits
a free action on a $\mathbb{Z}^n$-tree for an appropriate $n \in \mathbb{N}$, where $\mathbb{Z}^n$
is ordered lexicographically (see \cite{KMIrc}). The proof comes in several steps.
The initial breakthrough is due to Lyndon, who introduced a construction of the  free  $\Zt$-completion $\FZt$ of a free group $F$ (nowadays it is called Lyndon's free
$\Zt$-group) and showed that this group, as well as all its subgroups,  is fully residually free \cite{Lyndon2}.
Much later Remeslennikov proved that every finitely generated fully residually free group has a free Lyndon length function with values in $\mathbb{Z}^n$ (but not necessarily  ordered lexicographically) \cite{Remeslennikov:1989}.  That was a first link between limit groups and free actions on $\mathbb{Z}^n$-trees. In 1995 Myasnikov and Remeslennikov showed that Lyndon free exponential group $\FZt$  has a free Lyndon length function with values in $\mathbb{Z}^n$ with  lexicographical ordering \cite{Myasnikov_Remeslennikov:1995} and stated a conjecture that every limit group embeds into $\FZt$. Finally, Kharlampovich and Myasnikov proved that every limit group $G$ embeds into $\FZt$  \cite{KMIrc}.

Below, following \cite{MRS2} we construct a free $\Zt$-valued
length function on $\FZt$ which combined with the result of Kharlampovich and Myasnikov mentioned above
 gives a free $\mathbb{Z}^n$-valued length function on a given  limit group $G$. There are various algorithmic applications of
these results which are based on the technique of infinite words and Stallings foldings techniques for subgroups of $\FZt$ (see
\cite{KMRS, KMS12}).

\subsection{Lyndon's free group $\FZt$}

Let $A$ be an associative unitary ring. A group $G$ is termed an {\em $A$-group} if it is equipped
with a function ({\em exponentiation}) $G \times A \rightarrow G$:
$$(g,\alpha) \rightarrow  g^\alpha $$
satisfying the following conditions for arbitrary $g,h \in G$ and $\alpha, \beta \in A$:
\begin{enumerate}
\item[(Exp1)] $g^1 = g, \ \ \ g^{\alpha + \beta} = g^{\alpha} g^{\beta}, \ \ \ g^{\alpha \beta} =
(g^{\alpha})^{\beta}$,
\item[(Exp2)] $g^{-1} h^{\alpha} g = (g^{-1} h g)^{\alpha}$,
\item[(Exp3)] if $g$ and $h$ commute, then $(gh)^{\alpha} = g^{\alpha} h^{\alpha}$.
\end{enumerate}
The axioms (Exp1) and (Exp2) were introduced originally by R. Lyndon in \cite{Lyndon2}, the
axiom (Exp3) was added later in \cite{Myasnikov_Remeslennikov:1994}. A homomorphism $\phi: G
\rightarrow H$ between two $A$-groups is termed an $A$-homomorphism if $\phi(g^\alpha) = \phi(g)^\alpha$
for every $g \in G$ and $\alpha \in A$. It is not hard to prove (see, \cite{Myasnikov_Remeslennikov:1994})
that for every group $G$ there exists an $A$-group $H$ (which is unique up to an $A$-isomorphism)
and a homomorphism $\mu: G \longrightarrow H$ such that for every $A$-group $K$ and every $A$-homomorphism
$\theta: G \longrightarrow K$, there exists a unique $A$-homomorphism $\phi : H \longrightarrow K$
such that $\phi \mu = \theta$. We denote $H$ by $G^A$ and call it the {\em $A$-completion} of $G$.

In \cite{Myasnikov_Remeslennikov:1996} an effective  construction of $\FZt$ was given
in terms of extensions of centralizers. For a group $G$ let  $S = \{C_i \mid i \in I\}$ be a set of
representatives of conjugacy classes of proper cyclic centralizers in $G$, that is, every proper
cyclic centralizer in $G$ is conjugate to one from $S$, and no two centralizers from $S$ are conjugate.
Then the HNN-extension
$$H = \langle \ G, s_{i,j} \ \ (i \in I, j \in \mathbb{N}) \ \mid \ [s_{i,j}, u_i] = [s_{i,j},s_{i,k}] =
1\, (u_i \in C_i, i \in I, j,k \in \mathbb{N})\ \rangle,$$
is termed an {\em extension of cyclic centralizers} in $G$. Now the group $\FZt$ is
isomorphic to the direct limit of the following infinite chain of groups:
\begin{equation}
\label{eq:FZt}
F = G_0 < G_1 < \cdots < G_n < \cdots < \cdots,
\end{equation}
where $G_{i+1}$ is obtained from $G_i$ by  extension of all cyclic centralizers in $G_i$.

\subsection{Limit groups embed into $\FZt$}
\label{subs:limit_gps_embed}

The following results illustrate the connection of limit groups and finitely generated
subgroups of $\FZt$.

\begin{theorem} \cite{KMIrc}
\label{th:embKM}
Given a finite presentation  of a finitely generated fully residually free group $G$ one can
effectively construct an embedding $\phi : G \rightarrow \FZt$ (by specifying the images of
the generators of $G$).
\end{theorem}

Combining Theorem \ref{th:embKM} with the result on the representation of $\FZt$ as a union of a
sequence of extensions of centralizers one can get the following theorem.

\begin{theorem} \cite{Imp}
\label{th:embKMR}
Given a finite presentation of a finitely generated fully residually free group $G$ one can
effectively construct a finite sequence of extensions of centralizers
$$F < G_1 < \cdots < G_n,$$
where $G_{i+1}$ is an extension of the  centralizer of some element $u_i \in G_i$ by an infinite
cyclic group $\Z$, and an embedding $\psi^\ast : G \rightarrow G_n$ (by specifying the images of
the generators of $G$).
\end{theorem}

Now Theorem \ref{th:embKMR} implies the following important corollaries.

\begin{corollary} \cite{Imp}
\label{co:embKMR_1}
For every freely indecomposable non-abelian finitely generated fully residually free group one can
effectively find a non-trivial splitting (as an amalgamated product or HNN extension) over a cyclic
subgroup.
\end{corollary}

\begin{corollary} \cite{Imp}
\label{co:embKMR_2}
Every finitely generated fully residually free group is finitely presented. There is an algorithm
that, given a presentation of a finitely generated fully residually free group $G$ and generators
of the subgroup $H$, finds a finite presentation for $H$.
\end{corollary}

\begin{corollary} \cite{Imp}
\label{co:embKMR_3}
Every finitely generated residually free group $G$ is a subgroup of a direct product of finitely
many fully residually free groups; hence, $G$ is embeddable into $\FZt \times \cdots \times \FZt$.
If $G$ is given as a coordinate group of a finite system of equations, then this embedding can be
found effectively.
\end{corollary}

Let $K$ be an HNN-extension of a group $G$ with associated subgroups $A$ and $B$. Then $K$ is called a
separated HNN-extension if for any $g\in G,\ A^g \cap B = 1$.

\begin{corollary} \cite{Imp}
\label{co:embKMR_4}
Let a group $G$ be obtained from a free group $F$ by finitely many centralizer extensions. Then
every finitely generated subgroup $H$ of $G$ can be obtained from free abelian groups of finite
rank by finitely many operations of the following type: free products, free products with abelian
amalgamated subgroups at least one of which is a maximal abelian subgroup in its factor, free
extensions of centralizers, separated HNN-extensions with abelian associated subgroups at least
one of which is maximal.
\end{corollary}

\begin{corollary} \cite{Imp,Groves_Wilton:2009}
\label{co:embKMR_5}
One can enumerate all finite presentations of finitely generated fully residually free groups.
\end{corollary}

\begin{corollary} \cite{KMIrc}
\label{co:embKMR_6}
Every finitely generated fully residually free group acts freely on some $\Z^n$-tree with
lexicographic order for a suitable $n$.
\end{corollary}

\subsection{Description of $\mathbb{Z}^n$-free groups}
\label{subs:description_Z^n}

Given two $\Zt$-free groups $G_1,\ G_2$ (with free Lyndon lengths functions $\ell _1$ and $\ell _2$) and maximal abelian subgroups $A \leqslant G_1,\ B \leqslant
G_2$ such that
\begin{enumerate}
\item[(a)] $A$ and $B$ are cyclically reduced, 
\item[(b)] there exists an isomorphism $\phi : A \rightarrow B$ such that $\ell _2(\phi(a)) = \ell _1(a)$ for any $a \in A$.
\end{enumerate}
Then we call the amalgamated free product
$$\langle G_1, G_2 \mid A \stackrel{\phi}{=} B \rangle$$
the {\em length-preserving amalgam} of $G_1$ and $G_2$.

Given a $\Zt$-free group $H$ and non-conjugate maximal abelian subgroups $A, B \leqslant H$ such that \begin{enumerate}
\item[(a)] $A$ and $B$ are cyclically reduced,
\item[(b)] there exists an isomorphism $\phi : A \rightarrow B$ such that $\ell(\phi(a)) = \ell(a)$ and
$a$ is not conjugate to $\phi(a)^{-1}$ in $H$ for any $a \in A$.
\end{enumerate}
Then we call the HNN extension
$$\langle H, t \mid t^{-1} A t = B \rangle$$
the {\em length-preserving separated HNN extension} of $H$.

We now get the description of
regular $\Z^n$-free groups in the following form.

\begin{theorem}\cite{KMRS}
\label{th:complete_desrc}
A finitely generated group $G$ is regular $\Z^n$-free if and only if it can be obtained from free
groups by finitely many length-preserving separated HNN extensions and centralizer extensions.
\end{theorem}

\begin{theorem}\cite{KMS12}
\label{th:Z^n_desrc}
A finitely generated group $G$ is $\Z^n$-free if and only if it can be obtained from free groups by 
a finite sequence of length-preserving amalgams, length-preserving separated HNN extensions, and 
centralizer extensions.
\end{theorem}

Using this description it was proved in \cite{BK} that $\mathbb Z ^n$-free groups are $CAT (0)$.

\section{Products of trees}
\label{sec:products}

\subsection{Lattices from square complexes}

Lattices in products of trees provide examples for many interesting group theoretic properties,  for example there are finitely presented infinite simple groups \cite{burger-mozes:simple}, 
\cite{rattaggi:simple}
and many are not residually finite \cite{wise1}. 
For a great survey of  results of Burger, Mozes, Zimmer
on simple infinite groups acting on products of trees see \cite{mozes:survey}.

Torsion free lattices that are acting simply transitively on the vertices of the product of trees (not interchanging the factors) are fundamental groups of square complexes with just one vertex, a complete bipartite link and a VH-structure. There are many of such lattices, see \cite{stix-av} for a mass formula, but very rarely these lattices arise from an arithmetic context. The main purpose of this chapter is
to concentrate on the arithmetic case, mentioning other cases as well  as a mass formula for the relevant square complexes.

In this section we give a quick introduction to the geometry of square complexes and fix the terminology.

\subsubsection{Square complexes and products of trees}

A square complex $S$ is a $2$-dimensional combinatorial cell complex: its $1$-skeleton consists of a graph $S^1 = (V(S),E(S))$ with set of vertices $V(S)$, and set of oriented edges $E(S)$. The $2$-cells of the square complex come from a set of squares $S(S)$ that are combinatorially glued to the graph $S^1$ as explained below. Reversing the orientation of an edge $e \in E(S)$ is written as $e \mapsto e^{-1}$ and the set of unoriented edges is the quotient set 
\[
\ov{E}(S) = E(S)/(e \sim e^{-1}).
\]

More precisely, a square $\square$ is described by an equivalence class of $4$-tuples of oriented edges $e_i \in E(S)$
\[
\square = (e_1,e_2,e_3,e_4)
\]
where  the origin of $e_{i+1}$ is the terminus of $e_i$ (with $i$ modulo $4$). Such $4$-tuples describe the same square if and only if they belong to the same orbit under the dihedral action  generated by cyclically permuting the edges $e_i$ and by  the reverse orientation  map 
\[
(e_1,e_2,e_3,e_4) \mapsto (e_4^{-1},e_3^{-1},e_2^{-1},e_1^{-1}).
\]
We think of a square shaped $2$-cell glued to the (topological realization of the) respective edges of the graph $S^1$.
For more details on square complexes we refer for example to \cite{burger-mozes:lattices}. Examples for square complexes are given by products of trees.

\begin{remark}
We note, that in our definition of a square complex each square is determined by its boundary.
The group actions considered in the present survey are related only to such complexes.
\end{remark}

Let $T_n$ denote the $n$-valent tree. The product of trees 
\[
M = T_{m} \times T_{n}
\]
is a Euclidean building of rank $2$ and a square complex.  Here we are interested in  lattices, i.e., groups $\Gamma$ acting discretely and cocompactly on $M$ respecting the structure of square complexes.  The quotient $S = \Gamma \backslash M$ is a finite square complex, typically with orbifold structure coming from the stabilizers of cells.

We are especially interested in the case where $\Gamma$ is torsion free and acts simply transitively on the set of vertices of $M$. These yield the smallest quotients $S$ without non-trivial orbifold structure. Since $M$ is a CAT(0) space, any finite group stabilizes a cell of $M$ by the Bruhat--Tits fixed point lemma. Moreover, the stabilizer of a cell is pro-finite, hence compact, so that a discrete group $\Gamma$ acts with trivial stabilizers on $M$ if and only if $\Gamma$ is torsion free.

Let $S$ be a square complex. For $x \in V(S)$, let $E(x)$ denote the set of oriented edges originating in $x$. 
The \textbf{link} at the vertex $x$ in $S$ is the (undirected multi-)graph ${\mathbb Lk}_x$ whose set of vertices is $E(x)$ and whose set of edges in ${\mathbb Lk}_x$ joining vertices $a,b \in E(x)$ are given by corners $\gamma$ of squares in $S$ containing the respective edges of $S$,  see \cite{burger-mozes:lattices}. 

A covering space of a square complex admits a natural structure as a square complex such that the covering map preserves the combinatorial structure. In this way, a connected square complex admits a universal covering space.

\begin{proposition}
The universal cover of a finite connected square complex is a product of trees if and only if the link $ {\mathbb Lk}_x$ at each vertex $x$ is a complete bipartite graph.
\end{proposition}
\begin{proof}
This is well known and follows, for example, from \cite{Ballmann-Brin1995} Theorem C.
\end{proof}

\subsubsection{VH-structure}  \label{sec:VHstructure}
A  {\em vertical/horizontal structure}, in short a {\em VH-structure}, on a square complex $S$ consists of a bipartite structure $\ov{E}(S) = E_v \sqcup E_h$ on the set of unoriented edges of $S$ such that for every vertex $x \in S$ the link ${\mathbb Lk}_x$ at $x$ is the complete bipartite graph on the induced partition of $E(x) = E(x)_v \sqcup E(x)_h$. Edges in $E_v$ (resp.\ in $E_h$) are called vertical (resp.\ horizontal) edges. See \cite{wise1} for general facts on VH-structures. The {\em partition size} of the VH-structure is the function  $V(S) \to 
\mathbb N \times \mathbb N$ on the set of vertices 
\[
x \mapsto (\# E(x)_v, \# E(x)_h)
\]
or just the corresponding tuple of integers if the function is constant. Here $\#(-)$ denotes the cardinality of a finite set. 

We record the following basic fact, see \cite{burger-mozes:lattices} after Proposition 1.1:

\begin{proposition} \label{prop:VHstructureanduniversalcover}
Let $S$ be a square complex. The following are equivalent.
\begin{enumerate}
\item 
The universal cover of $S$ is a product of trees $T_m \times T_n$ and the group of covering transformations does not interchange the factors.
\item 
There is a VH-structure on $S$ of constant partition size $(m,n)$. \hfill $\square$
\end{enumerate}
\end{proposition}

\begin{corollary} \label{cor:latticesversusVH}
Torsion free cocompact lattices $\Gamma$ in  $\mathbb Aut(T_m) \times \mathbb Aut(T_n)$ not interchanging the factors and 
up to conjugation correspond uniquely to finite square complexes with a VH-structure of partition size $(m,n)$ up to isomorphism.
\end{corollary}
\begin{proof}
A lattice $\Gamma$ yields a finite square complex $S = \Gamma \backslash T_m \times T_n$ of the desired type. Conversely, a finite square complex $S$ with VH-structure of constant partition size $(m,n)$ has universal covering space $M = T_m \times T_n$ by Proposition~\ref{prop:VHstructureanduniversalcover}, and the choice of a base point vertex $\tilde{x} \in M$ above the vertex $x \in S$  identifies $\pi_1(S,x)$ with the lattice $\Gamma = \mathbb Aut(M/S) \subseteq \mathbb Aut(T_m) \times \mathbb Aut(T_n)$. The lattice depends on the  chosen base points only up to conjugation.
\end{proof}

Simply transitive torsion free lattices not interchanging the factors as in Corollary~\ref{cor:latticesversusVH} correspond to square complexes with only one vertex and a VH-structure, necessarily of constant partition size. These will be studied in the next section.

\subsubsection{$1$-vertex square complexes}

Let $S$ be a square complex with just one vertex $x \in S$ and a VH-structure $\ov{E}(S) = E_v \sqcup E_h$. Passing from the origin to the terminus of an oriented edge induces a fixed point free involution on $E(x)_v$ and on $E(x)_h$. Therefore the partition size is necessarily a tuple of even integers.

\begin{definition} \label{defi:BMstructureingroup}
A  {\em vertical/horizontal structure}, in short {\em VH-structure},  in a group $G$ is an ordered pair $(A,B)$ of finite subsets $A,B \subseteq G$ such that the following holds.
\begin{enumerate}
\item \label{defiitem:AandBinvolution} Taking inverses induces fixed point free involutions on $A$ and $B$.
\item The union $A \cup B$ generates $G$.
\item \label{defiitem:ABequalsBA} The product sets $AB$ and $BA$ have size $\#A \cdot \#B$ and $AB = BA$.
\item \label{defiitem:AB2torsion} The sets $AB$ and $BA$ do not contain $2$-torsion.
\end{enumerate}
The tuple $(\#A,\#B)$ is called the {\em valency vector} of the VH-structure in $G$.
\end{definition}


Similar as in  \cite{burger-mozes:lattices} \S6.1, starting from a VH-structure the following construction 
\begin{equation} \label{eq:constructionSAB}
(A,B) \leadsto S_{A,B}
\end{equation}
yields a square complex $S_{A,B}$ with one vertex and VH-structure starting from a VH-structure $(A,B)$ in a group $G$. The vertex set $V(S_{A,B})$ contains just one vertex $x$. The set of oriented edges of $S_{A,B}$ is the disjoint union 
\[
E(S_{A,B}) = A \sqcup B
\]
with the orientation reversion map given by $e \mapsto e^{-1}$. Since $A$ and $B$ are preserved under taking inverses, there is a natural vertical/horizontal structure such that $E(x)_h = A$ and $E(x)_v = B$. 
The squares of $S_{A,B}$ are constructed as follows. Every relation in $G$ 
\begin{equation}\label{eq:relationabba} 
ab = b'a'
\end{equation}
with $a,a' \in A$ and $b,b' \in B$ (not necessarily distinct) gives rise to a $4$-tuple 
\[
(a,b,a'^{-1},b'^{-1}).
\] 
The following relations are equivalent to \eqref{eq:relationabba} 
\begin{eqnarray*}
a'b^{-1} & = &  b'^{-1}a, \\
a^{-1}b' & = & ba'^{-1}, \\
a'^{-1}b'^{-1} & = & b^{-1}a^{-1}.
\end{eqnarray*}
and we consider the four $4$-tuples so obtained
\[
(a,b,a'^{-1},b'^{-1}), \quad (a',b^{-1},a^{-1},b'), \quad (a^{-1},b',a',b^{-1}), \quad (a'^{-1},b'^{-1},a,b)
\]
as equivalent. A square $\square$ of $S_{A,B}$ consists of an equivalence class of such $4$-tuples arising from a relation of the form \eqref{eq:relationabba}.

\begin{lemma}
The link ${\mathbb Lk}_x$ of $S_{A,B}$ in $x$ is the complete bipartite graph $L_{A,B}$ with vertical vertices labelled by $A$ and horizontal vertices labelled by $B$.
\end{lemma}
\begin{proof}
By \ref{defiitem:ABequalsBA} of Definition~\ref{defi:BMstructureingroup} every pair $(a,b) \in A \times B$ occurs on the left hand side  in a relation of the form \eqref{eq:relationabba} and therefore the link ${\mathbb Lk}_x$ contains $L_{A,B}$.

If \eqref{eq:relationabba} holds, then the set of left hand sides of equivalent relations 
\[
\{ab, a'b^{-1},a^{-1}b',a'^{-1}b'^{-1}\}
\]
is a set of cardinality $4$, because $A$ and $B$ and $AB$ do not contain $2$-torsion by Definition~\ref{defi:BMstructureingroup} \ref{defiitem:AandBinvolution} + \ref{defiitem:AB2torsion} and the right hand sides of the equations are unique by 
Definition~\ref{defi:BMstructureingroup} \ref{defiitem:ABequalsBA}. Therefore $S_{A,B}$ only contains $(\#A \cdot \#B)/4$ squares. It follows that ${\mathbb Lk}_x$ has at most as many edges as $L_{A,B}$, and, since it contains $L_{A,B}$, must agree with it.
\end{proof}

\begin{definition}
We will call the complex $S_{A,B}$ a $(\#A,\#B)$-complex, keeping in mind, that there are many complexes with the same valency vector. Also, if a group is a fundamental group of 
a $(\#A,\#B)$-complex, we will call it a $(\#A,\#B)$-group.
\end{definition}

\subsection{Mass formula for one vertex square complexes with VH-structure}  \label{sec:mass_formula}
In this section we present a mass formula for the number of one vertex square complexes with VH-structure up to isomorphism where each square complex is counted with the inverse order of its group of automorphisms as its weight \cite{stix-av}.

Let $A$ (resp.\ $B$) be a set with fixed point free involution of size $2m$ (resp.\ $2n$). 
In order to count one vertex square complexes $S$ with VH-structure with vertical/horizontal partition $A \sqcup B$ of oriented edges  we introduce the generic matrix
\[
X = (x_{ab})_{a\in A, b \in B}
\]
with rows indexed by $A$ and columns indexed by  $B$ and with $(a,b)$-entry a formal variable $x_{ab}$. Let $X^t$ be the transpose of $X$, let $\tau_A$ (resp.\ $\tau_B$) be the permutation matrix for $e \mapsto e^{-1}$ for $A$ (resp.\ $B$). For a square $\square$  we set 
\[
x_\square = \prod_{e \in \square} x_e
\]
where the product ranges over the edges $e = (a,b)$ in the link of $S$ originating from $\square$ and $x_e = x_{ab}$. 
Then the sum of the $x_\square$, when $\square$ ranges over all possible squares with edges from $A \sqcup B$, reads 
\[
\sum_{\square} x_\square =  \frac{1}{4}\tr\big((\tau_A X \tau_B X^t\big)^2),
\]
and the number of  one vertex square complexes $S$ with VH-structure of partition size $(2m,2n)$ and edges labelled by $A$ and $B$  is given by 
\begin{equation} \label{eq:labeled_mass_formula}
\widetilde{{\rm BM}}_{m,n} = \frac{1}{(mn)!} \cdot \frac{\partial^{4mn}}{\prod_{a,b} \partial x_{ab} } \left( \frac{1}{4}\tr\big((\tau_A X \tau_B X^t\big)^2)\right)^{mn}. 
\end{equation}
Note that this is a constant polynomial.

We can turn this into a mass formula for the number of one vertex square complexes with VH-structure up to isomorphism where each square complex is counted with the inverse order of its group of automorphisms as its weight. We simply need to divide by the order of the universal relabelling 
\[
\#({\mathbb Aut}(A,\tau_A) \times {\mathbb Aut}(B,\tau_B)) = 2^n(n)! \cdot 2^m(m)!.
\]
Hence the mass of one vertex square complexes with VH-structure is given by
\begin{equation} \label{eq:mass_formula}
{\rm BM}_{m,n} = \frac{1}{2^{n+m+2nm}(n)! \cdot (m)! \cdot (mn)!} \cdot \frac{\partial^{4mn}}{\prod_{a,b} \partial x_{ab} } \left( \tr\big((\tau_A X \tau_B X^t\big)^2)\right)^{mn}.
\end{equation}
The formula \eqref{eq:labeled_mass_formula} reproduces the numerical values of $\widetilde{{\rm BM}}_{m,n}$ for small values $(2m,2n)$ that were computed by Rattaggi in \cite{rattaggi:thesis} table B.3. Here small means $mn \leq 10$.

\subsection{Arithmetic groups acting on products of two trees}

There is a deep connection between arithmetic lattices in products of trees and quaternion algebras. 
For background on quaternion algebras see \cite{vigneras}.

For any ring $R$ we consider the $R$-algebra
$$\mathbb{H}(R)=\{a=a_0 +a_1{\bf i}+a_2{\bf j}+a_3{\bf k} ; a_1,a_2,a_3,a_4 \in R\},$$
with $R$-linear multiplication given by ${\bf i}{\bf j}={\bf k}=-{\bf j}{\bf i}$, ${\bf i}^2={\bf j}^2=-1$.
An example of such an algebra are classical Hamiltonian quaternions $\mathbb{H}(\mathbb{R})$.
Recall, that the (reduced) norm is a homomorphism

$$|\cdot|: \mathbb{H}(R)^\times \to R^\times, ||a||= a_0^2+a_1^2+a_2^2+a_3^2.$$

Let $\mathbb{H}(\mathbb{Z}) $ be the integer quaternions.
Let $S_p$ be the set of integer quaternions 
$$a=a_0 + a_1{\bf i}+a_2{\bf j} + a_3{\bf k} \in \mathbb{H}(\mathbb{Z}) $$
with $a_0>0$, $a_0$ odd, and $|a|^2=p$, so the reduced norm of $a$ is $p$. Then, by a result of Jacobi,
$\#S_p=p+1.$

\medskip

If $p\equiv 1(\rm mod\ 4)$ is prime, 
then $x^2 \equiv -1(\rm mod\ p)$ has a solution in $\mathbb{Z}$,
so, by Hensel's Lemma, $x^2=-1$ has a solution $i_p$ in $\mathbb{Q}_p$, where
$\mathbb{Q}_p$ is the field of $p$-adic numbers.

\medskip

Define the following homomorphism

\[
\psi_p : \mathbb{H}(\mathbb{Z}[\frac{1}{p}])^\times \to PGL_2(\mathbb{Q}_p).
\]
by
$$\psi_p(a) = \left(\begin{array}{rr}
a_0 + a_1i_p & a_2 + a_3i_p \\
-a_2+ a_3i_p & a_0 - a_1i_p
\end{array}\right)$$

\begin{theorem}[\cite{lps}]

$\psi_p(S_p)$ contains $p+1$ elements and generates the free group $\Gamma _p$
of rank $(p+1)/2$.

\end{theorem}

$\Gamma_p$ acts freely and transitively on the vertices of the $(p+1)$-regular
tree $T_{p+1}$.

\bigskip


\[
\psi_{p,l} : \mathbb{H}(\mathbb{Z}[\frac{1}{pl}])^\times \to PGL_2(\mathbb{Q}_p)\times PGL_2(\mathbb{Q}_l).
\]

by
$$\psi_{p,l}(a) = \left[\begin{pmatrix}
a_0 + a_1i_p & a_2 + a_3i_p \\
-a_2+ a_3i_p & a_0 - a_1i_p
\end{pmatrix}, \begin{pmatrix}
a_0 + a_1i_l & a_2 + a_3i_l \\
-a_2+ a_3i_l & a_0 - a_1i_l
\end{pmatrix}\right],$$

where $$a=a_0 + a_1{\bf i}+a_2{\bf j} + a_3{\bf k} \in \mathbb{H}(\mathbb{Z}), $$ $p,l \equiv 1(\rm mod\ 4)$ are two distinct primes 
and $i_p \in \mathbb{Q}_p$, $i_l \in \mathbb{Q}_l $ satisfy the conditions
$i^2_p+1=0$ and $i^2_l+1=0$, as above.

Let $S_p$ be as above and $S_l$ be the set of integer quaternions 
$$a=a_0 + a_1{\bf i}+a_2{\bf j} + a_3{\bf k} \in 
\mathbb{H}(\mathbb{Z}[\frac{1}{pl}]) $$
with $a_0>0$, $a_0$ odd, and $|a|^2=l$. Then $S_l=l+1.$

Let $\Gamma_{p,l}$ be the subgroup of $PGL_2(\mathbb{Q}_p)\times PGL_2(\mathbb{Q}_l)$ 
generated by $\psi(S_p \cup S_l)$.

Mozes has proved the following result:

\begin{theorem}[\cite{mozes}] If $p,l \equiv 1(\rm mod\ 4)$ are two distinct prime numbers, then
$$\Gamma_{p,l}<PGL_2(\mathbb{Q}_p)\times PGL_2(\mathbb{Q}_l)< Aut(T_{p+1})\times Aut(T_{l+1})$$
is a $(p+1,l+1)$-group (see definition 3 above).

\end{theorem}

Rattaggi \cite{rattaggi:thesis} extended this construction to primes
$p \equiv 3(\rm mod\ 4)$ using solutions in $\mathbb{Z}$ of $x^2+y^2\equiv -1 (\rm mod\ p)$.

\bigskip

Note, that the constructions of Mozes and Rattaggi can not be extended for products
of trees of the same valency. However, some applications of arithmetic lattices, like
constructions of fake quadrics in algebraic geometry, do require arithmetic lattices acting
on products of trees of the same valency, see \cite{stix-av} for more detailed motivation.

Now we describe arithmetic lattices acting simply transitively on products of trees
of the same valency. They turn out to be related to quaternion algebras in finite characteristic
and are arithmetic lattices in $\PGL_2(\mathbb{F}_q((t))) \times \PGL_2(\mathbb{F}_q((t)))$.

 Let $D$ be a quaternion algebra over a global function field $K/\mathbb{F}_q$
 of a smooth curve over $\mathbb{F}_q$. 
Let $S$ consist of the ramified places of $D$ together with two distinct unramified $\mathbb{F}_q$-rational places $\tau$ and $\infty$.
An $S$-arithmetic subgroup $\Gamma$ of the projective linear 
group $G = PGL_{1,D}$ of $D$ acts as a cocompact lattice
 on the product of trees  $T_{q+1} \times T_{q+1}$ that are
 the Bruhat--Tits buildings for $G$ locally at $\tau$ and $\infty$.
There are plenty of such arithmetic lattices, however, in general the action of $\Gamma$ on the set of vertices will not be simply transitive.

Let $q$ be an odd prime power and let $c \in \mathbb{F}_q^\ast$ be a non-square. 
Consider the $\mathbb{F}_q[t]$-algebra with non-commuting variables $Z,F$
$B = \mathbb{F}_q[t]\{Z,F\}/(Z^2 = c, F^2 = t(t-1), ZF = -FZ)$,
an $\mathbb{F}_q[t]$-order in the quaternion algebra $D = B \otimes \mathbb{F}_q(t)$ over $\mathbb{F}_q(t)$ ramified in $t=0,1$.
Let $G = PGL_{1,B}$
be the algebraic group over $\mathbb{F}_q[t]$ of units in $B$ modulo the center.

The following results give first examples of lattices acting on products of trees in finite characteristic.

\begin{theorem}[\cite{stix-av}]

Let $q$ be an odd prime power,  and choose a generator $\delta$ of the cyclic group $\mathbb{F}_{q^2}^\ast$. Let 
$\tau \not= 1$ be an element of $\mathbb{F}_q^\ast$, and
 $\zeta \in \mathbb{F}_{q^2}^\ast$ be an element with 
norm $\zeta^{1+q} = (\tau-1)/\tau$. Let $G/\mathbb{F}_q[t]$ be the algebraic group of G.
The irreducible arithmetic lattice $G(\mathbb{F}_q[t,\frac{1}{t(t-\tau)}])$ has
 the following presentation: 
\begin{eqnarray*}
G(\mathbb{F}_q[t,\frac{1}{t(t-\tau)}])  & \simeq & \left\langle \  d,a,b  \ \left| 
\begin{array}{c}
d^{q+1} =  a^2 = b^2 = 1 \\[0.3ex]
(d^i a d^{-i})(d^{j} b d^{-j}) = (d^l b d^{-l})(d^{k} a d^{-k})  \\[0.3ex]
\text{ for all } 0 \leq i,l \leq q \text{ and } j,k \text{ determined by  } (\star) 
\end{array} 
\right.
\right\rangle,
\end{eqnarray*}
where $(\star)$ is the system of equations in  the quotient group 
$\mathbb{F}_{q^2}^\ast/\mathbb{F}_q^\ast$

$\delta^{j-l} =  (1 -  \zeta \delta^{(i-l)(1-q)}) \cdot  \delta^{(q+1)/2} \ \text{ and } \
\delta^{k-i}  =  (1 -  \frac{1}{\zeta \delta^{(i-l)(1-q)}}) \cdot \delta^{(q+1)/2}.$

\end{theorem}

We are now prepared to establish presentations for the arithmetic lattices with VH-structures
in finite characteristic.

\begin{theorem}[\cite{stix-av}]
Let $1 \not= \tau \in {\mathbb F}_q^\ast$, let $c \in {\mathbb F}_q^\ast$ be a non-square and let  ${\mathbb F}_q[Z]/{\mathbb F}_q$ be the quadratic field extension  with $Z^2 = c$. The group $\Gamma_\tau$  is a torsion free arithmetic lattice in $G$ with finite presentation 
\[
\Gamma_\tau = \left\langle 
\begin{array}{c}
{a}_\xi, {b}_\eta \text{ for } \xi,\eta \in \bF_q[Z]  \text{ with } \\[1ex]
N(\xi) = -c \text{ and } N(\eta) = \frac{c\tau}{1-\tau}  
\end{array}
\left|
\begin{array}{c}
a_\xi a_{-\xi} = 1, \ b_\eta b_{-\eta} = 1 \\[1ex]
\text{ and }  a_\xi b_\eta = b_\lambda a_\mu   \text{ if in } \bF_q[Z]: \\[1ex]
\eta = \lambda^q(\lambda + \xi)^{1-q}  \text{ and } \mu = \xi^q (\xi+\lambda)^{1-q}
\end{array} 
\right.
\right\rangle
\]
which acts  simply transitively via the Bruhat Tits action on the vertices of $T_{q+1} \times T_{q+1}$,
so the $\Gamma_\tau$ groups are $(q+1,q+1)$-groups.
\end{theorem}

\begin{corollary}[\cite{stix-av}]
The number of commensurability classes of arithmetic lattices acting simply transitively on a product
of two trees grows at least linearly on $q$.

\end{corollary}

The following group is the smallest example of a torsion free arithmetic lattice acting on a product of two trees.

\[
\Gamma_2 = \left\langle g_0,g_1, g_2, g_3  \ \left| \ 
g_0g_1^{-1}  = g_1g_2, \ g_0g_3^{-1} = g_3g_2^{-1}, \ g_0g_3 = g_1^{-1}g_0^{-1}, \ g_2g_3 = g_1g_2^{-1} 
\right. \right\rangle
\]
where, in the notation of Theorem 37 with a generator $\delta$ of the cyclic group $\bF_3[Z]^\times$, we have $g_i = a_{\delta^i}$ for $i$ even and $g_i = b_{\delta^i}$ for $i$ odd.

\subsection{Fibered products of square complexes}

The fibered product of two 1-vertex square complexes was defined
in \cite{burger-mozes:lattices} in the following way:

Let 1-vertex square complex $X$ be given by the data: $A1, A2$ and 
$S  \subset  A1 \times A2 \times A1 \times A2 $
then the fibered product $Y$ is defined by the data: $A1 \times A1, A2 \times A2$ and 
$R = \{((a1,a2),(b1,b2),(a1^{\prime},a2^{\prime}),(b1^{\prime},b2^{\prime})) :
 (ai,bi,ai^{\prime},bi^{\prime}) \in S, i = 1,2\}$.
 
 It was shown in \cite{BMZ} that if the fundamental group of $X$ is just infinite
 (no proper normal subgroups of infinite index), then the fundamental group of $Y$
 is not residually finite.
 
 This result creates a machinery to construct many non-residually finite groups using
 arithmetic lattices. In particular, all families of arithmetic groups from the previous section
 generate families of non-residually finite groups.

\section{Lattices in products of $n \geq 3$ trees}

\subsection{Arithmetic higher-dimensional lattices}

Similar to products of two trees, there are many arithmetics lattices in products of
$n \geq 3$ trees, see, for example, \cite{livne},
but it is difficult to get simply transitive actions.

We start with the following example of a group acting simply transitively on a product of three trees,
with 9 generators and 26 relations.

$G= \{a1,a2,b1,b2,b3,c1,c2,c3,c4 | a1*b1*a2*b2,a1*b2*a2*b1^{-1},a1*b3*a2^{-1}*b1,a1*b3^{-1}*a1*b2^{-1},a1*b1^{-1}*a2^{-1}*b3,a2*b3*a2*b2^{-1},a1*c1*a2^{-1}*c2^{-1},a1*c2*a1^{-1}*c3,a1*c3*a2^{-1}*c4^{-1},a1*c4*a1*c1^{-1},a1*c4^{-1}*a2*c2,
a1*c3^{-1}*a2*c1,a2*c3*a2*c2^{-1},a2*c4*a2^{-1}*c1,c1*b1*c3*b3^{-1},c1*b2*c4*b2^{-1},c1*b3*c4^{-1}*b2,
c1*b3^{-1}*c4*b3,c1*b2^{-1}*c2*b1,c1*b1^{-1}*c4*b1^{-1},c2*b2*c3^{-1}*b3^{-1},c2*b3*c4*b1,
c2*b3^{-1}*c3*b3,c2*b2^{-1}*c3*b2,c2*b1^{-1}*c3*b1^{-1},c3*b1*c4*b2\}$.

The group $G$ is an arithmetic lattice acting simply transitively on a product of three trees,
and a particular case of a $n$-dimensional construction of arithmetic lattices,
described below. 

\begin{theorem}[\cite{sv}]
If $p_1,...,p_n$ are n distinct prime numbers, then there is an arithmetic lattice
$$\Gamma_{p_1,...,p_n}<PGL_2(\mathbb{Q}_{p_1})\times...\times PGL_2(\mathbb{Q}_{p_n})< Aut(T_{p_1+1})\times... \times Aut(T_{p_n+1})$$
acting simply transitively on the product of $n$ trees.
\end{theorem}

\begin{proof}

The group of $S$-integral quaternions (in the example above $S = \{3,5,7\})$, let's say with $n$ 
equals the cardinality of $S$, should have the 
following structure: there are sets of generators $S_p$ for every $p \in 
 S$ and relations of type $ab=b^{\prime}a^{\prime}$ for every pair of distinct primes $p$, 
$\ell \in S$ with $a,a^{\prime}\in S_p $ and $b,b^{\prime} \in S_\ell$. That's because for 
every subset $T \subseteq S$ the $T$-integral quaternions form a subgroup 
of the $S$-integral ones.
By work of Mozes \cite{mozes} and Rattaggi  \cite{rattaggi:thesis} we know that $S_p \cup 
S_\ell$ generates the $\{p,\ell\}$-integral quaternions and that it acts 
simply transitively on vertices of a product of trees $T_{p+1} \times T_{\ell + 1}$.
We denote this group $\Gamma_{p,l}$ for future references.

 Since these quaternions are integral at all other primes, this 
subgroup fixes the standard vertex of the (Bruhat-Tits)-tree 
for that component.

 This means that the group generated by all the $S_p$'s for $p \in S$ will 
map the standard vertex to all its neighbours, hence by homogeneity 
and connectedness the action is simply transitive on vertices also on the product  
$T_{p_1} \times ... \times T_{p_n}$ where $S = \{p_1, ..., p_n\}$.  
We then may conclude that the action is simple transitive on 
vertices. This identifies the arithmetic group with the combinatorial 
 fundamental group of the cube complex, whose $\pi_1$ is generated by 
edges (these are the $S_p$'s for direction "p") and relations all come 
 from the 2-dim cells, hence the squares, all of which are also squares 
 in one of the subgroups with just two primes. Because we know all 
squares for these subgroups, we know all relations for the S-integral 
quaternions. These relations give the relations for the group $\Gamma_{p_1,...,p_n}$.
In the example above, the group $G$, acting on a product of three trees
has three 2-dimensional subgroups $\Gamma_{3,5}$, $\Gamma_{5,7}$, $\Gamma_{3,7}$.

\end{proof}

Another example of a group acting cocompactly on a product
of more than two trees is the lattice of Glasner-Mozes \cite{glasner-mozes}, its construction is inspired by
Mumford's fake projective plane.

\subsection{Non-residually finite higher-dimensional lattices}

We give an elementary construction of a family of non-residually finite groups acting
simply transitively on products of $n \geq 3$ trees (for more details see \cite{sv}).

\begin{theorem}[\cite{sv}]
There exist a family of non-residually finite groups acting simply transitively
on a product of $n \geq 3$ trees.
\end{theorem}

\begin{proof} This is just a sketch.
We start with an arithmetic lattice $\Gamma$ acting simply transitively on a product of $n$
trees. There are $n$ 2-dimensional groups induced by $\Gamma$ in a unique way.
For each of them we apply the fibered product construction and take the union of all
generators and relations of each fibered product. The resulting group is non-residually finite
by \cite{BMZ}.
\end{proof}
\section{Open problems}

We will mention here some open problems and conjectures.

{\bf $\Lambda$-free groups.}

\begin{conjecture} \cite{KMS12} Every finitely generated $\Lambda$-free group is finitely presented.
\end{conjecture}
This would imply that all finitely generated $\Lambda$-free groups are $\mathbb R ^n$-free.

\begin{problem}\cite{KMS12}
\label{th:main2}
Is any  finitely presented $\Lambda$-free group $G$  also $\mathbb{Z}^k$-free for an appropriate $k \in
\mathbb{N}$ and lexicographically ordered $\mathbb{Z}^k$?
\end{problem}

{\bf Arithmetic groups.}

\begin{conjecture} (Caprace)
The number of non-commensurable arithmetic lattices acting simply transitively on product
of two trees is bounded polymonially on $q$.
\end{conjecture}
\begin{problem}
What is the structure and properties of groups acting on products of three and more trees.
\end{problem}

\end{document}